\documentclass[final,a4paper]{article} 
\usepackage{amsthm}
\usepackage{amsmath}
\usepackage{amssymb}
\usepackage{esint}
\usepackage{amsmath,amssymb,amsfonts}
\usepackage{color}
\usepackage{verbatim}
\usepackage{eepic,epic,epsfig,epstopdf}
\usepackage[]{graphics}
\usepackage{graphicx,inputenc}
\usepackage{graphicx,subfig}
\graphicspath{{./pics/}}
\usepackage{bm,color,url}
\usepackage{algorithm,algorithmic}
\usepackage{xr}
\usepackage{float}
\usepackage{setspace}
\newtheorem{assumption}{Assumption}[section]
\newtheorem{proposition}{Proposition}[section]
\newtheorem{exam}{Example}[section]

\newtheorem{lemma}{Lemma}[section]
\newtheorem{theorem}{Theorem}[section]
\newtheorem{definition}{Definition}[section]
\newtheorem{remark}{Remark}[section]
\theoremstyle{definition}
\usepackage{geometry}
\def\sgn{{\mathrm{sgn}}}
\def\calA{{\mathcal{A}}}
\def\calI{{\mathcal{I}}}

\title{A Unified Primal Dual Active Set Algorithm for Nonconvex Sparse Recovery}
\author{Jian Huang\thanks{Department of Applied Mathematics, The Hong Kong Polytechnic University, Hong Kong. (j.huang@polyu.edu.hk)}
\and  Yuling Jiao\thanks{School of Statistics and Mathematics, Zhongnan University of Economics and Law, Wuhan, 430063, P.R. China. (yulingjiaomath@whu.edu.cn)}
\and Bangti Jin\thanks{Department of Computer Science, University College London, Gower Street, London WC1E 6BT, UK. (b.jin@ucl.ac.uk, bangti.jin@gmail.com)}
\and Jin Liu\thanks{Center of Quantitative Medicine Duke-NUS Medical School, Singapore. (jin.liu@duke-nus.edu.sg)}
\and Xiliang Lu\thanks{School of Mathematics and Statistics, Wuhan University, Wuhan 430072, P.R. China, and
Hubei Key Laboratory of Computational Science (Wuhan University), Wuhan, 430072, China. (xllv.math@whu.edu.cn)}
\and Can Yang\thanks{Department of Mathematics, The Hong Kong University of Science and Technology, Clear Water Bay, Hong Kong. (macyang@ust.hk)}
}
\begin{document}

\maketitle

\begin{abstract}
In this paper, we consider the problem of recovering a sparse signal based on penalized least squares
formulations. We develop a novel algorithm of primal-dual active set type for a class of
nonconvex sparsity-promoting penalties, including $\ell^0$, bridge, smoothly clipped absolute
deviation, capped $\ell^1$ and minimax concavity penalty. First
we establish the existence of a global minimizer for the related optimization problems. Then
we derive a novel necessary optimality condition for the global minimizer using the associated
thresholding operator. The solutions to the optimality system are coordinate-wise
minimizers, and under minor conditions, they are also local minimizers. Upon introducing
the dual variable, the active set can be determined using the primal and dual variables together. Further, this
relation lends itself to an iterative algorithm of active set type which at each step involves
first updating the primal variable only on the active set and then updating the dual variable
explicitly. When combined with a continuation strategy on the regularization parameter,
the primal dual active set method is shown to converge globally  to the underlying regression target
under certain regularity conditions. Extensive numerical experiments with both simulated and real data
demonstrate its superior performance in efficiency and accuracy compared with the existing sparse recovery methods.
\bigskip
\noindent
\textbf{Running title: UPDAS for nonconvex sparse recovery}
\end{abstract}

\section{Introduction}\label{sec:intro}

In this paper, we develop a fast algorithm of primal dual active set (PDAS) type for a class of
nonconvex optimization problems arising in sparse recovery. Sparse recovery is a fundamentally
important problem in statistics, machine learning and signal processing. In statistics, sparsity
is one vital variable selection tool for constructing parsimonious models that admit easy
interpretation \cite{Tibshirani:1996}. In signal processing, especially compressive sensing,
sparsity represents an important structural property that can be effectively exploited for data
acquisition, signal transmission, storage and processing etc \cite{CandesTao:2005,Donoho:2006}.
Generally, the forward model is formulated as
\begin{equation}\label{model}
  y = \Psi x^{\dag} + \eta,
\end{equation}
where the vector $x^{\dag} \in \mathbb{R}^{p}$ denotes the regression coefficient or
the signal to be recovered, the vector $\eta \in
\mathbb{R}^{n} $ is the random error term, and the matrix $\Psi\in\mathbb{R}^{n\times p}$ is a design matrix or model
describing the system response mechanism in signal processing. Throughout, we assume that the matrix $\Psi$
has normalized column vectors $\{\psi_i\}$, i.e., $\|\psi_i\| = 1$, $i=1,...,p$, where $\|\cdot\|$
denotes the Euclidean norm of a vector. When
$n\ll p$, problem \eqref{model} is severely underdetermined (and ill-posed), and hence it is challenging
to obtain a meaningful solution. The sparsity approach looks for a solution with many zero
entries, and it opens a novel avenue for resolving the issue.

One popular method for realizing sparsity constraints is basis pursuit \cite{Chen:1998}
or lasso \cite{Tibshirani:1996}. It leads to a convex but nonsmooth optimization problem:
\begin{equation}\label{reglasso}
 \min_{x \in \mathbb{R}^{p}}  \tfrac{1}{2}\|\Psi x
 -y\|^{2} + \lambda\|x\|_{1},
\end{equation}
where $\|\cdot\|_1$ denotes the $\ell^1$-norm of a vector, and $\lambda>0$ is a regularization parameter.
Since its introduction \cite{Chen:1998,Tibshirani:1996}, problem \eqref{reglasso} has gained immense
popularity in many diverse disciplines, which can largely be attributed to the fact that problem \eqref{reglasso} admits
efficient numerical solution. The convexity of the problem allows designing fast and globally convergent
minimization algorithms, e.g., gradient projection method and coordinate descent algorithm; see
\cite{TroppWright:2010} for an overview. Theoretically, minimizers to \eqref{reglasso} enjoy attractive
statistical properties \cite{ZhaoYu:2006,CandesTao:2005,MeinshausenBuhlmann:2006}. In particular, under certain regularity
conditions (e.g., restricted isometry property and restricted eigenvalue condition) on the design matrix
$\Psi$ and the sparsity level of the true signal $x^\dag$, it can produce models with good estimation and
prediction accuracy, and also the support of the true signal can be correctly identified with a high probability
\cite{ZhaoYu:2006}.

However, it is also well known that the convex model \eqref{reglasso} has some drawbacks: it requires more
restrictive conditions on the matrix $\Psi$ and more data in order to recover exactly the signal than
nonconvex ones, e.g., bridge penalty \cite{ChartrandStaneva:2008,FoucartLai:2009,Sun:2012}; and it tends
to produce biased estimates for large coefficients \cite{ZhangHuang:2008}, and hence lacks oracle property
\cite{FanLi:2001,FanPeng:2004}. To circumvent these drawbacks, a number of nonconvex penalties have been proposed,
including $\ell^{0}$, bridge \cite{FrankFriedman:1993,Fu:1998}, capped-$\ell^{1}$ \cite{Zhang:2010b},
smoothly clipped absolute deviation (SCAD) \cite{FanLi:2001} and minimax concave penalty (MCP)
\cite{Zhang:2010a} etc; see Section \ref{ssec:nonconvex} for details.

The nonconvex approach leads to the following optimization problem:
\begin{equation}\label{eqn:model}
  \min_{x\in\mathbb{R}^p} J(x) = \tfrac{1}{2}\|\Psi x - y\|^2 + \sum_{i=1}^p \rho_{\lambda,\tau}(x_i),
\end{equation}
where $\rho_{\lambda,\tau}$ is a nonconvex penalty, $\lambda>0$ is a regularization parameter,
and $\tau$ controls the {degree of} concavity of the penalty (see Section
\ref{ssec:nonconvex} for details). The nonconvexity and nonsmoothess of the penalty $\rho_{\lambda,
\tau}$ poses significant challenge in their mathematical analysis and efficient numerical solutions.
Nonetheless, their attractive theoretical properties \cite{ZhangZhang:2012} and empirical
successes have generated much interest in developing efficient and accurate numerical algorithms.

\subsection{Literature overview on algorithms for nonconvex sparse recovery}

In this part, we provide an overview about existing algorithms for some popular nonconvex penalties.
First we briefly survey specialized methods for the $\ell^0$, bridge, SCAD and MCP, separately.

First, for the $\ell^0$ penalty, iterative hard thresholding is very popular \cite{Kingsbury:2001,
BlumensathDavies:2008}. The iterates generated by the algorithm are descent to the objective functional and
converge to a local minimizer, with an asymptotic linear convergence, if the matrix $\Phi$ satisfies certain
conditions \cite{BlumensathDavies:2008}. {There are a number of closely related}
iterative methods, e.g., orthogonal matching pursuit  \cite{TroppGilbert:2007} and CoSaMP
\cite{NeedellTropp:2009}, PDAS \cite{JiaoJinLu:2015} and SDAR \cite{HuangJiaoLiuLu:2017}.
Algorithmically, these methods all exploit the  dual and / or primal information
to adaptively update the signal support, and each step involves one least-squares problem on
the support only, which allows significantly reducing the computational complexity for sparse solutions.
Further, mixed integer programming was adopted  for the $\ell^0$ penalty \cite{Bertsimas:2016,LiuYaoLi:2016},
which is also applicable to MCP and SCAD.

Second, for the bridge penalty, one popular idea is to use the iteratively reweighted
least-squares method together with suitable smoothing of the singularity at the origin
\cite{ChartrandYin:2008,LaiWang:2011}. In the works \cite{LaiWang:2011,LaiXuYin:2013}, the
convergence of the iterates to a critical point of the smoothed functional was established; see
\cite{ZhaoLi:2012} for an alternative scheme and its convergence analysis. In the work \cite{Lu:2012},
a unified convergence analysis was provided, and new variants were also developed.
Each iterate of the method in \cite{LaiWang:2011} and \cite{ZhaoLi:2012} respectively
requires solving a penalized least-squares problem and a weighted lasso problem, which can
be fairly expensive for high-dimensional problems. For the bridge penalty, one can
also employ an iterative thresholding algorithm, for which the iterates converge subsequentially,
and the limit satisfies a necessary optimality condition \cite{BrediesLorenz:2009}.

Third and last, for the SCAD, Fan and Li  \cite{FanLi:2001} proposed to use a local quadratic
approximation (LQA) to the nonconvex penalty  and one single Newton step for optimizing
the resulting functional. Later a local linear approximation (LLA) was suggested
\cite{ZouLi:2008} to replace the LQA, which leads to a one-step sparse estimator.
For the closely related MCP, Zhang \cite{Zhang:2010a}
developed an algorithm that keeps track of multiple local minima in order to select a
solution with desirable statistical properties.

In the literature, there are also several general-purposed algorithms that aim at treating the model
\eqref{eqn:model} in a unified framework, including majorization-minimization, iterative thresholding,
coordinate descent, DC programming, and proximal gradient method etc. The first algorithm is based on the idea of majorization-minimization,
where each step involves a reweighted $\ell^1$ or $\ell^2$ subproblem, and includes the LLA and LQA
for SCAD and multi-stage convex relaxation \cite{Zhang:2010b} for the smoothed bridge and the
capped $\ell^1$ penalty. Numerically, the cost per iteration is that of the $\ell^2$/$\ell^1$
solver, and thus can be expensive. The subproblems may be solved
approximately, e.g., with one or several gradient steps, in order to enhance the computational efficiency.
Theoretically, the sequence of iterates is descent for the functional, but the convergence of the sequence
itself is generally unclear.  In the work \cite{GongZhangLuHuangYe:2013}, a general iterative shrinkage
and thresholding algorithm was developed, and the convergence to a critical point was shown under a
coercivity assumption on the objective functional. The bridge, SCAD, MCP, capped $\ell^1$ and log-sum penalties
were demonstrated. The second algorithm is based on coordinate descent, which at each step updates one
component of the signal vector in either Jacobi \cite{She:2009} or Gauss-Seidel \cite{BrehenyHuang:2011,
MazumderFriedmanHastie:2011} fashion for the SCAD and MCP.
Theoretically, any cluster point of the iterates is a stationary point \cite{Tseng:2001}.
Numerical experiments \cite{MazumderFriedmanHastie:2011} also verified its efficiency for such penalties.
Third, in the work \cite{GassoRakotomanonjyCanu:2009}, an algorithm was proposed based on decomposing the nonconvex
penalty into the difference of two convex functions and then applying DC programming. The authors
illustrated the idea on the bridge, SCAD, and capped-$\ell^1$ penalty. Like the first algorithm, each
iteration involves a convex weighted lasso problem (and thus can be expensive). Fourth, the path-following proximal gradient method
was proposed for MCP, SCAD and capped-$\ell^1$ \cite{WangLiuZhang:2014, LohWainwright:2015}.
Fifth and last, Chen et al \cite{ChenNiuYuan:2013} (see also \cite{Chen:2012}) derived affine-scaled second-order
necessary and sufficient conditions for local minimizers to the model \eqref{eqn:model} in the case of the
bridge, SCAD and MCP, and developed a globally convergent smoothing trust-region Newton method.
Meanwhile, in \cite{HintermullerWu:2012} a superlinearly convergent regularized Newton method was developed.

\subsection{Contributions}

The main contributions of this work are three-folded.

First, we establish the existence of a global minimizer to problem \eqref{eqn:model}; see Proposition \ref{thm:Exist}. To
the best of our knowledge, the existence issue has not been thoroughly studied in prior works. In this work, we also derive
a necessary optimality condition of global minimizers to \eqref{eqn:model} using the associated thresholding operator, and
prove that any solution to the necessary optimality condition is a coordinate-wise minimizer. Further, we study the relation
between a coordinate-wise minimizer and local minimizer to problem \eqref{eqn:model},  and provide numerically verifiable
sufficient conditions for a coordinate-wise minimizer to be a local minimizer in Theorem \ref{thm:localmin}. These results
represent the essential theoretical contributions of this work.

Second, inspired by the necessary optimality condition of coordinate-wise minimizer, we develop a UPDAS method with
continuation (UPDASC) to approximate the solution path  for problem \eqref{eqn:model} with all five popular  nonconvex
penalties listed in Table \ref{tab:nonconvex}. The algorithm is straightforward to implement. Further, we propose a new
tuning parameter selection rule, which couples seamlessly with the continuation strategy without any extra computational
cost. We prove the solutions along the path converge globally to the underlying regression coefficient  under certain
regularity conditions on the design matrix $\Psi$; see Theorem \ref{thm:convergence}. This represents the main algorithmic innovation of the work.

Third and last, we conduct extensive numerical experiments with both simulated and real data to demonstrate the efficiency
and accuracy of UPDASC, as well as the feasibility of the proposed tuning parameter selection rule. In particular, our methods are several
times faster than \texttt{glmnet}, which is one of the fastest Lasso solvers currently available. The \texttt{MATLAB} and \texttt{R} packages are available at
the following links \url{http://www0.cs.ucl.ac.uk/staff/b.jin/software/updasc.zip} and  \url{https://github.com/gordonliu810822/PDAS}, respectively.

Now we put the present work in the context of statistical estimation.
Recently, there is an important line of ongoing research aiming at bounding the estimation error of some popular algorithms
for high-dimensional nonconvex penalized regression to close the gap between statistics and computation,
e.g.,  multi-stage convex relaxation \cite{Zhang:2010b}, DC programming \cite{WangKimLi:2013}, LLA \cite{FanXueZou:2014}, and
path-following proximal gradient method \cite{WangLiuZhang:2014, LohWainwright:2015}. The proposed UPDASC is along this line
of research. i.e., it aims at finding a good approximation of the true signal. In particular, the convergence (consistence) result
in Theorem \ref{thm:convergence} is in the sense of statistics, i.e., convergence to the underlying regression coefficient,
instead of in the sense of optimization, where convergence to local (global) minimizers of a given objective function is
of major interest. The afore-mentioned prior works consider only first-order methods, while the proposed UPDASC is a Newton type method.
Generally, designing and analyzing fast and stable Newton type algorithms for high dimension penalized regression remain a very
challenging task. In the prior works \cite{LiSunToh:2016,HuangJiaoLuYang:2016} and  \cite{JiaoJinLu:2015,HuangJiaoLiuLu:2017},
Newton type methods have been developed for lasso and $\ell^0$ problems, respectively. UPDASC in the present work is a unified framework of
Newton type method  to handle general nonconvex penalized regression, which represents an important step forward along
the research direction, and holds significant potential for nonconvex sparse recovery.

\subsection{Organization of the paper}

The rest of the paper is organized as follows. In Section \ref{sec:setting}, we  describe the nonconvex penalties
and establish the existence of a global minimizer to problem \eqref{eqn:model}. In Section \ref{sec:op},
we first derive the thresholding operator for each penalty, and then use it in the necessary optimality
condition, whose solutions are coordinate-wise minimizers to problem \eqref{eqn:model}. Further, we give
sufficient conditions for a coordinate-wise minimizer to be a local minimizer. In Section \ref{sec:alg},
by introducing a dual variable, we rewrite the necessary optimality condition
and the active set using both primal and dual variables. Based on this fact, we develop a unified PDAS algorithm for all five
nonconvex penalties. Further, we establish the global convergence of the
algorithm when it is coupled with a continuation strategy. Finally, numerical results for several
examples are presented in Section \ref{sec:num} to illustrate the efficiency and accuracy of the algorithm.
The proofs of the theoretical results can be found in the supplementary materials.

\section{Problem formulation}\label{sec:setting}

In this section, we specify explicitly the nonconvex penalties of interest, and
discuss the existence of a global minimizer to problem \eqref{eqn:model}.

\subsection{Nonconvex penalties}\label{ssec:nonconvex}

We focus on five commonly used nonconvex penalties, {i.e., $\ell^0$, bridge, SCAD, MCP and capped $\ell^1$,}
for recovering sparse signals; see Table \ref{tab:nonconvex} for the explicit formulas (and the associated thresholding
operators, to be defined below). Next we briefly review these nonconvex penalties.

\begin{table}
  \centering
  \caption{Nonconvex penalty functions $\rho_{\lambda,\tau}(t)$ and the thresholding
  operators $S_{\lambda,\tau}^\rho(v)$. In the Table, $(t^*,T^*)$ and $G(v)$ are given in
  Lemma \ref{thm:tT} and the proof of Proposition \ref{thm:thresholding}, cf. Appendix
  \ref{app:thresholding}.}\label{tab:nonconvex}{\scalebox{.9}{
  \begin{tabular}{lcc}
    \hline
     penalty & $\rho_{\lambda,\tau}(t)$ & $S^\rho_{\lambda,\tau}(v)$\\
    \hline
    lasso \cite{Tibshirani:1996,Chen:1998}  & $\lambda|t|$  & $\sgn(v)\max(|v|-\lambda,0)$ \\
    $\ell^0$ \cite{Akaike:1974}  & $\left\{\begin{array}{ll}\lambda & t\neq 0 \\[1.0ex]
     0 & t=0 \end{array}\right.$ & $\left\{\begin{array}{ll} 0 & |v| < \sqrt{2\lambda} \\[1.0ex]
        \{0, \sgn(v)\sqrt{2\lambda}\} & |v| = \sqrt{2\lambda} \\[1.1ex]
      v & |v| > \sqrt{2\lambda} \end{array}\right.$\\
    bridge, $0<\tau<1$ \cite{FrankFriedman:1993,Fu:1998}& $\lambda|t|^\tau$ &  $\left\{\begin{array}{ll} 0 & |v| < T^* \\[1.1ex]
          \{0, \sgn(v)t^*\} & |v| = T^* \\ [1.1ex]
        \mathop\textrm{argmin}\limits_{u>0}\,G(u) & v > T^*\\[1.1ex]
     -S^{\ell^\tau}_{\lambda,\tau}(-v) & v<-T^* \end{array}\right.$\\
    capped-$\ell^1$, $\tau > \frac{1}{2}$ \cite{Zhang:2010b} & $\left\{\begin{array}{ll}
      \lambda^2\tau & |t|>\lambda\tau\\[1.1ex]
       \lambda |t|& |t|\leq\lambda\tau\end{array}\right.$  &  $\left\{\begin{array}{ll}0 &|v| \leq \lambda \\[1.1ex]
      \sgn(v)(|v| - \lambda) &\lambda <|v| < \lambda(\tau + \frac{1}{2})\\[1.1ex]
      \sgn(v)(\lambda\tau \pm \frac{\lambda}{2}) & |v| = \lambda(\tau + \frac{1}{2})\\[1.1ex]
       v & |v|> \lambda(\tau + \frac{1}{2})
      \end{array}\right.$\\
    SCAD, $\tau >2$ \cite{FanLi:2001} &  $\left\{\begin{array}{ll}\frac{\lambda^2(\tau+1)}{2} & |t|>\lambda\tau\\[1.1ex]
     \frac{\lambda\tau|t|-\frac{1}{2}(t^2+\lambda^2)}{\tau-1} & \lambda<|t|\leq \lambda\tau\\[1.1ex]
      \lambda|t| & |t|\leq\lambda\end{array}\right.$   & $\left\{\begin{array}{ll}0 & |v| \leq \lambda \\[1.1ex]
     \sgn(v)(|v| - \lambda) &\lambda <|v| \leq 2\lambda \\[1.1ex]
     \sgn(v)\frac{(\tau-1)|v| - \lambda\tau}{\tau -2} & 2\lambda < |v| \leq \lambda\tau\\[1.1ex]
     v & |v| > \lambda\tau \end{array}\right.$\\
    MCP, $\tau>1$ {\cite{Zhang:2010a}} & $\left\{\begin{array}{ll}
      \lambda(|t|-\tfrac{t^2}{2\lambda\tau}) & |t|<\tau\lambda\\[1.1ex]
       \tfrac{\lambda^2\tau}{2} &|t|\geq\tau\lambda \end{array}\right.$ &  $\left\{\begin{array}{ll}
             0 & |v| \leq \lambda\\[1.1ex]
             \sgn(v)\frac{\tau(|v| - \lambda)}{\tau -1} &\lambda \leq |v| \leq \lambda\tau \\[1.1ex]
             v & |v| \geq \lambda\tau\end{array}\right.$\\
    \hline
  \end{tabular}}}
\end{table}

The $\ell^0$-norm, denoted by $\|x\|_0$ of a vector $x$, is defined by $\|x\|_0 = |\{i: x_i\neq0\}|$. It
penalizes the number of nonzero components, which measures the model complexity (e.g., degree of freedom).
Due to the discrete nature of the $\ell^0$ penalty, the model \eqref{eqn:model} is combinatorial in nature and hardly
tractable in high-dimensional spaces (see, e.g., \cite{ChenGeWang:2017} for the NP hardness). All other penalties in Table
\ref{tab:nonconvex} can be regarded as approximations to the $\ell^0$ penalty, and are designed to alleviate
its drawbacks, e.g., lack of stability \cite{Breiman:1996} and computational challenges.

The bridge penalty was popularized by the works \cite{FrankFriedman:1993,Fu:1998}. The $\ell^\tau$-quasinorm
$\|x\|_\tau$, $0<\tau<1$, of a vector $x$, defined by $\|x\|_\tau^\tau=\sum_i|x_i|^\tau$, is a
quasi-smooth approximation of the $\ell^0$ penalty as $\tau$ tends towards zero \cite{ItoKunisch:2011},
and related statistical properties, e.g., variable selection and oracle property, have been intensively studied
\cite{KnightFu:2000,HuangHorowitzMa:2008,ChartrandStaneva:2008,FoucartLai:2009}.

SCAD \cite{FanLi:2001,FanPeng:2004} was suggested to circumvent the drawbacks of lasso. It was devised
based on the following qualitative requirements: the penalty is singular at the origin in order to
achieve sparsity and its derivative vanishes for large values so as to ensure unbiasedness. Specifically,
for SCAD, it is defined for $\tau>2$ via
\begin{equation*}
  \rho_{\lambda,\tau}(t) = \lambda\int_0^{|t|} \min\left(1,\frac{\max(0,\lambda\tau-|s|)}{\lambda(\tau-1)}\right){\rm d}s
\end{equation*}
and computing the integral explicitly yields the expression in Table \ref{tab:nonconvex}. Further,
variable selection consistency and asymptotic estimation efficiency
were studied in \cite{FanPeng:2004}.

The capped-$\ell^1$ penalty \cite{Zhang:2010b} is a linear approximation of the SCAD penalty.
Theoretically, it can be viewed as a variant of the two-stage optimization problem:
one first solves a regular lasso problem and then solves a lasso problem where the large
coefficients are not penalized any more, thus leading to an unbiased model. The
condition $\tau>1/2$ ensures the well-posedness of the thresholding operator \cite{ZhangZhang:2012}.

The {MCP was devised in the same spirit as SCAD.} It  is defined by \cite{Zhang:2010a}
\begin{equation*}
  \begin{aligned}
    \rho_{\lambda,\tau}(t) & = \lambda\int_0^{|t|}\max\left(0,1-|s|/(\lambda\tau)\right){\rm d}s.
  \end{aligned}
\end{equation*}
The MCP minimizes the maximum concavity $\sup_{0<t_1<t_2}
(\rho'_{\lambda,\tau}(t_1)-\rho'_{\lambda,\tau}(t_2))/(t_2-t_1)$ subject to
unbiasedness and feature selection constraints: $\rho'_{\lambda,\tau}(t)=0$ for any
$|t|\geq\lambda\tau$ and $\rho'_{\lambda,\tau}(0^\pm)=\pm\lambda$.
Similar to the capped-$\ell^1$ penalty, the condition $\tau >1$ ensures the
well-posedness of the thresholding operator \cite{Zhang:2010a}.

\subsection{Existence of global minimizers}
To put the algorithmic developments on a firm theoretical foundation, we first consider
the existence of a global minimizer to the nonconvex functional $J$ defined in problem \eqref{eqn:model}.
The standard argument in calculus of variation for proving existence relies on the lower semi-continuity
and coercivity of the objective function, and in the absence of these properties, it is nontrivial to
prove the existence. First, we note that the $\ell^0$ penalty is lower semi-continuous \cite{ItoKunisch:2011}
and the rests are continuous.  Hence, if the matrix $\Psi$ is of full column rank, i.e., $\|\Psi x\|
\rightarrow\infty$ as $\|x\|\rightarrow \infty$, then the existence of a global minimizer follows by the
standard argument. However, in the setting of $p > n$, which is of interest in sparse recovery, $\Psi$
does not have a full column rank, and the standard argument does not apply directly. Moreover, $\ell^0$, capped-$\ell^1$,
SCAD and MCP penalties do not satisfy the coercivity. Consequently, the existence of a global minimizer
of the nonconvex functional $J$ is not self evident.

\begin{proposition}\label{thm:Exist}
For any of the five nonconvex penalties in Table \ref{tab:nonconvex}, there exists at least
one global minimizer to problem \eqref{eqn:model}.
\end{proposition}

This seemingly  simple result requires a careful argument, where the challenge lies mainly in the
lack of the coercivity, as mentioned above. The complete proof is given in the supplementary materials. In passing, we note
that the existence issue in the case of the $\ell^0$ penalty was discussed in \cite{Nikolova:2013}.
However, to the best of our knowledge, the existence issue in a general setting has not been studied for SCAD,
capped-$\ell^1$ penalty and MCP before. Note that the global minimizer is generally not unique.

\section{Necessary optimality condition for minimizers}\label{sec:op}
Now we derive the necessary optimality condition for global minimizers to \eqref{eqn:model}, which also
forms the basis for deriving the PDAS algorithm in Section \ref{sec:alg}. We shall show that the solutions
to the necessary optimality condition are coordinate-wise minimizers, and provide verifiable sufficient
conditions for a coordinate-wise minimizer to be a local minimizer.

\subsection{Thresholding operators}\label{ssec:thresholding}
First we derive thresholding operators for the penalties in Table \ref{tab:nonconvex}. The thresholding
operator forms the basis of many existing algorithms, e.g., coordinate descent and iterative thresholding,
and thus unsurprisingly the expressions in Table \ref{tab:nonconvex} were derived earlier (see e.g.
\cite{She:2009,MazumderFriedmanHastie:2011,BrehenyHuang:2011,ItoKunisch:2011,GongZhangLuHuangYe:2013}),
but in different manners. Below we shall provide a unified derivation and a useful characteristic of the
thresholding operator. To this end, for any penalty $\rho(t)$ in Table \ref{tab:nonconvex} (the subscripts
$\lambda$ and $\tau$ are omitted for simplicity), we define a function $g(t): [0,\infty)\to \mathbb{R}^+\cup\{0\}$ by
\begin{equation*}
g(t) = \left\{\begin{array}{ll}\frac{t}{2} + \frac{\rho(t)}{t}, & t \neq 0,\\[1.0ex]
   \liminf\limits_{t\rightarrow 0^+} g(t), & t=0. \end{array} \right.
\end{equation*}
\begin{lemma}\label{lem:gt}
The value $T^\ast = \inf_{t > 0}g(t)$ is attained at some point $t^*\geq 0$.
\end{lemma}
\begin{proof}
By the definition of the function $g(t)$, it is continuous over the interval $(0,+\infty)$,
and approaches infinity as $t\to +\infty$. Hence any minimizing sequence $\{t_n\}$
is bounded. If the sequence contains a positive accumulation point $t^*$, then $g(t^*) = T^*$
by the continuity of $g$. Otherwise it has only an accumulation point $0$. However,
by the definition of $g(0)$, $g(0) = T^\ast$ and hence $t^\ast = 0$.
\end{proof}

The explicit expressions of the tuple $(t^*,T^*)$ for the penalties in Table \ref{tab:nonconvex}
are given below; see Appendix \ref{app:tT} in the supplementary materials for the proof.
\begin{lemma}\label{thm:tT} For the five nonconvex penalties in Table \ref{tab:nonconvex}, there holds
\begin{equation*}{\small
  (t^*,T^*) = \left\{\begin{array}{ll}
    (\sqrt{2\lambda},\sqrt{2\lambda}), & \ell^0,\\
    ((2\lambda(1-\tau))^{\frac{1}{2-\tau}},(2-\tau)\left[2(1-\tau)\right]^{\frac{\tau-1}{2-\tau}}
     \lambda^{\frac{1}{2-\tau}}), & \ell^\tau,\\
    (0,\lambda), & \mathrm{capped}\mbox{-}\ell^1,\ \mathrm{SCAD},\ \mathrm{MCP}.
  \end{array}\right.}
\end{equation*}
\end{lemma}

Next we introduce the thresholding operator $S^\rho$ defined by
\begin{equation}\label{eqn:thresholding}
  S^\rho(v) =  \mathop\textrm{argmin}_{u\in \mathbb{R}}\left({(u-v)^2}/{2} + \rho(u)\right),
\end{equation}
which can potentially be set-valued. First we give a useful characterization of $S^\rho$
based on $(t^*, T^*)$.

\begin{lemma}\label{lem:proximal}
Let
$u^\ast \in \mathop\mathrm{arg}\min_{u\in \mathbb{R}}\left({(u-v)^2}/{2} + \rho(u)\right).$
Then the following three statements hold:
  $(\mathrm{a})$ $u^\ast=0 \Rightarrow |v|\leq T^{*}$;
  $(\mathrm{b})$ $|v|<T^{*} \Rightarrow u^\ast=0$;
  and $(\mathrm{c})$ $|v|=T^{*} \Rightarrow u^\ast=0$  or  $g(u^\ast)=\sgn(v) T^{*}$.
\end{lemma}

If the minimizer $t^*$ to $g(t)$ is unique, then assertion (c) of Lemma \ref{lem:proximal}
can be replaced by $|v|=T^{*} \Rightarrow u^\ast=0$  or  $u^\ast=\mathrm{sgn}(v) t^{*}$.

Now we can derive an explicit expression for the thresholding operator $S^\rho$,
which is summarized in Table \ref{tab:nonconvex} and given by Proposition \ref{thm:thresholding}
below. The proof is elementary but lengthy, and thus deferred to Appendix \ref{app:thresholding}.
\begin{proposition}\label{thm:thresholding}
The thresholding operators $S^\rho$ associated with the five nonconvex penalties {\rm(}$\ell^0$, bridge,
capped-$\ell^1$, SCAD and MCP{\rm)} are as given in Table \ref{tab:nonconvex}.
\end{proposition}

Note that the thresholding operator $S^\rho$ is singled-valued, except at $v = T^*$ for
the $\ell^\tau$, $0\leq \tau <1$, penalty, and at $v = \lambda(\tau + \frac{1}{2})$ for
the capped-$\ell^1$ penalty.
\subsection{Necessary optimality condition}

Now we derive the necessary optimality condition for a global minimizer to \eqref{eqn:model}
using the thresholding operator $S^\rho$. To this end, we first
recall the concept of coordinate-wise minimizers. Following \cite{Tseng:2001}, a vector $x^*=(x_1^*,x_2^*,
\dots,x_p^*)\in\mathbb{R}^p$ is called a coordinate-wise minimizer of the functional $J(x)$ if it is the
minimum along each coordinate direction, i.e.,
\begin{equation}
x_i^* \in \mathop\mathrm{arg}\min\limits_{t\in\mathbb{R}}  J(x_1^*,...,x_{i-1}^*,t,x_{i+1}^*,...,x_p^*).
\end{equation}

Next we derive the sufficient and necessary optimality condition for a coordinate-wise minimizer $x^*$ of problem
\eqref{eqn:model}. By the definition of $x^*$, there holds
\begin{equation*}
{\small
  \begin{aligned}
    &x_i^* \in \mathop\textrm{argmin}\limits_{t\in \mathbb{R}} J(x_1^*,...,x_{i-1}^*,t,x_{i+1}^*,...,x_p^*)\\
    \Leftrightarrow\quad & x_i^* \in \mathop\textrm{argmin}\limits_{t\in \mathbb{R}} \tfrac{1}{2}\|\Psi x^*-y+(t-x_i^*)\psi_i\|^2+\rho_{\lambda,\tau}(t) \\
    \Leftrightarrow\quad & x_i^* \in \mathop\textrm{argmin}\limits_{t\in \mathbb{R}} \tfrac{1}{2}(t-x_i^*)^2+(t-x_i^*)\psi_i^t(\Psi x^*-y) + \rho_{\lambda,\tau}(t) \\
   \Leftrightarrow\quad & x_i^* \in \mathop\textrm{argmin}\limits_{t\in \mathbb{R}} \tfrac{1}{2}(t-x_i^*-\psi_i^t(y-\Psi x^*))^2 + \rho_{\lambda,\tau}(t).
  \end{aligned}}
\end{equation*}
By introducing the dual variable $d_i^* = \psi_i^t ( y -\Psi x^*)$ and recalling
the definition of the thresholding operator $S_{\lambda,\tau}^\rho(t)$ for $\rho_{\lambda,\tau}$,
we have the following characterization of $x^*$,
which clearly is also a necessary optimality condition of a global minimizer.
\begin{lemma}\label{lem:necopt}
An element $x^*\in\mathbb{R}^p$ is a coordinate-wise minimizer to problem \eqref{eqn:model} if and only if
\begin{equation}\label{eqn:nec}
   x_i^* \in S_{\lambda,\tau}^\rho (x_i^* + d_i^*) \quad \mbox{ for } i=1,...,p,
\end{equation}
where the dual variable $d^*$ is defined by $d^* = \Psi^t(y-\Psi x^*)$.
\end{lemma}

\begin{remark}{
In the same manner, we can derive the well-known necessary and sufficient KKT condition for lasso \cite{Combettes:2005}.}
\end{remark}

Using the expression of the thresholding operators in Table \ref{tab:nonconvex} and
the remark following Proposition \ref{thm:thresholding},
only in the case of $|x_i^* + d_i^*| = T^*$ for the bridge and $\ell^0$ penalties, and $|x_i^*
+ d_i^*| = \lambda(\tau + \frac{1}{2})$ for the capped-$\ell^1$ penalty, the value of the entry
$x_i^*$ is not uniquely determined.

The necessary optimality condition \eqref{eqn:nec} forms the basis of the PDAS algorithm in
Section \ref{sec:alg}. Hence, the ``optimal solution'' by the algorithm can at best solve
the necessary condition, and it is important to study more precisely the meaning of ``optimality''.
First we recall a well-known result. By \cite[Lemma 3.1]{Tseng:2001}, a coordinate-wise minimizer
$x^*$ is a stationary point in the following sense
\begin{equation}
   \liminf\limits_{t\rightarrow 0^+} \frac{J(x^*+td) - J(x^*)}{t} \geq 0, \quad \forall d\in\mathbb{R}^p.
\end{equation}

In general, a coordinate-wise minimizer $x^*$ is not necessarily a local minimizer, i.e.,
$J(x^*+\omega) \geq J(x^*)$ for all small $\omega\in\mathbb{R}^p$. Below
we provide sufficient conditions for a coordinate-wise minimizer to be a local minimizer.
To this end, we denote by $\calA=\{i:\ x_i^*\neq0\}$ and $\calI=\calA^c$ the active and
inactive sets, respectively, of a coordinate-wise minimizer $x^*$. Throughout, for any subset $\mathcal{A}
\subset \mathbb{I} = \{1,2,...,p\}$, we use the notation $x_\mathcal{A}\in\mathbb{R}^{|\mathcal{A}|}$
(or $\Psi_\mathcal{A}\in\mathbb{R}^{n\times |\mathcal{A}|}$) for the subvector of $x$ (or the submatrix of $\Psi$)
consisting of entries (or columns) whose indices are listed in $\mathcal{A}$.

For any $\calA\subset \mathbb{I}$, let $\sigma(\calA)$ be the smallest singular value of matrix $\Psi_\calA^t\Psi_\calA$.
Then
\begin{equation}\label{eqn:sec}
  \|\Psi_\calA x_\calA\|^2 \geq \sigma(\calA)\|x_\calA\|^2.
\end{equation}
Intuitively, the condition \label{equ:sec} ensures that the smooth convex term dominates the
nonsmooth nonconvex term so that the coordinatewise minimizer $x^*$ has good property.
The sufficient conditions for a coordinatewise minimizer to be a local minimizer are summarized in Theorem \ref{thm:localmin} below. The proof
is lengthy and technical, and hence deferred to Appendix \ref{app:localmin}. Under
the prescribed conditions, the solution generated by the PDAS algorithm, if it does
converge, is a local minimizer.
\begin{theorem}\label{thm:localmin} Let $x^*$ be a coordinate-wise minimizer to \eqref{eqn:model}, and $\calA=
\{i: x_i^*\neq0\}$ and $\calI=\calA^c$ be the active and inactive sets, respectively. Then there hold:
  \begin{itemize}
    \item[$(\mathrm{i})$] $\ell^0$: $x^*$ is a local minimizer.
    \item[$(\mathrm{ii})$] bridge: If $\sigma(\mathcal{A})
     >\frac{\tau}{2}$ in \eqref{eqn:sec}, then $x^*$ is a local minimizer.
    \item[$(\mathrm{iii})$] capped-$\ell^1$: If $\{i: \ |x_i^*|=\lambda\tau\}=\emptyset$, then $x^*$ is a local minimizer.
    \item[$(\mathrm{iv})$] SCAD: If $\sigma(\mathcal{A})
      > \frac{1}{\tau-1}$ in \eqref{eqn:sec} and $\|d^*_\mathcal{I}\|_{\infty} < \lambda$, then $x^*$ is a local minimizer.
    \item[$(\mathrm{v})$] MCP: If $\sigma(\mathcal{A})
      > \frac{1}{\tau}$ in \eqref{eqn:sec} and $\|d^*_\mathcal{I}\|_{\infty}< \lambda$, then $x^*$ is a local minimizer.
  \end{itemize}
\end{theorem}

This theorem shows that, for the $\ell^0$ penalty, a coordinate-wise minimizer
is always a local minimizer.
For the capped-$\ell^1$ penalty, the sufficient condition $\{i: \ |x_i^*|=\tau
\lambda\}=\emptyset$ is related to the nondifferentiability of $\rho_{\lambda,\tau}^{c\ell^1}(t)$
at $t=\tau\lambda$. For the bridge, SCAD and MCP, the condition \eqref{eqn:sec} is essential
for a coordinate-wise minimizer to be a local minimizer, which requires that the size of the
active set be not large. The condition $\|d^*_\mathcal{I}\|_{\infty} < \lambda$
is closely related to the uniqueness of the global minimizer. 
If both $\Psi$ and $\eta$ are random Gaussian, it holds except a null
measure set \cite{Zhang:2010a}.

The conditions in Theorem \ref{thm:localmin} involve only the computed solution and the parameters
$\lambda$ and $\tau$, and are numerically verifiable, which in principle enables one to check
\textit{a posteriori} whether a coordinatewise minimizer is a local one.

\begin{remark}
{For MCP and SCAD, we can prove that a local minimizer is also a coordinatewise minimizer,
in view of the convexity of the one-dimensional minimization problem \eqref{eqn:thresholding}. Generally, the
regularity condition $\sigma(\mathcal{A})$ being bounded away from $0$ in \eqref{eqn:sec} cannot be removed,
in order to ensure the coordinatewise minimizer to be a local minimizer;
See Appendix \ref{counterexample} for a counterexample.}
\end{remark}

\section{Primal dual active set algorithm}\label{sec:alg}
In this section, we propose an algorithm of PDAS type for the nonconvex penalties
listed in Table \ref{tab:nonconvex}, discuss its efficient implementation via a
continuation strategy, and analyze its global convergence.

\subsection{Brief review on semismooth Newton method and PDAS algorithm}

First, we briefly review semismooth Newton methods and primal dual active set algorithm,
following the monograph \cite{ItoKunisch:2008}. Let $X$ and $Z$ be Banach spaces and consider
the following nonlinear equation
\begin{equation}\label{eqn:nonlin}
 F(x) = 0 ,
\end{equation}
where $F: D \subset X \to Z$, and $D$ is an open subset of $X$. The semismooth Newton
method builds on the concept of a generalized derivative known as Newton derivative.
The notation $\mathcal{L}(X,Z)$ denotes the space of bounded linear operators from $X$ to $Z$.
\begin{definition}{\cite{Kummer:1988,HintermullerItoKunisch:2003}}
The mapping $F : D\subset X \to Z$ is called Newton differentiable in
the open subset $U \subset D$ if there exists a family of mappings
$G: U \to \mathcal{L}(X,Z)$ such that
\begin{equation*}
\lim_{\|h\|\to 0}\frac{\|F(x+h)-F(x)-G(x+h)h\|}{\|h\|} = 0,\quad \forall x\in U.
\end{equation*}
The mapping $G$ is called a Newton derivative for $F$ in $U$.
\end{definition}

Note that $G$ is not required to be unique to be a Newton derivative for $F$ in $U$. Under the assumption
of Newton differentiability in an open set, Newton's method converges superlinearly
for appropriate choices of the initialization; see the following convergence result \cite{ChenNashedQi:2000}.
\begin{proposition}
Suppose that $x^*$ is a solution to \eqref{eqn:nonlin} and that $F$ is Newton
differentiable in an open neighborhood $U$ containing $x^*$ with
Newton derivative $G(x)$. If $G(x)$ is nonsingular for all $x\in U$ and
$\{\|G(x)^{-1}\| : x\in U\}$ is bounded, then the Newton iteration
\begin{equation}\label{eqn:Newton}
 x^{k+1} = x^k - G(x^k)^{-1}F(x^k)
\end{equation}
converges superlinearly to $x^*$, provided that $\|x^0 - x^*\|$ is sufficiently small.
\end{proposition}

It is well known within the optimal control community that many PDAS type methods can be interpreted as
a semismooth Newton method, upon choosing a proper Newton derivative \cite{HintermullerItoKunisch:2003}.
Thus, PDAS algorithms merit fast local convergence. We illustrate the equivalence with the lasso
problem \eqref{reglasso}, which was developed in several works \cite{GriesseLorenz:2008,FanJiaoLu:2014,
HuangJiaoLuYang:2016,LiSunToh:2016}. Recall that the KKT system of the lasso problem \eqref{reglasso}
is given by  $\Psi^t\Psi x + d  = \Psi^ty$ and $x = S_\lambda(x + d)$ \cite{Combettes:2005}, where
$S_\lambda$ is the soft thresholding operator for the $\ell^1$ penalty. Then we introduce a nonlinear
operator $F:\mathbb{R}^{2p}\to\mathbb{R}^{2p}$ by
\begin{equation*}
  F(x,d) = \left[\begin{array}{c}
     \Psi^t\Psi x + d - \Psi^ty\\
     x - S_\lambda(x + d)
   \end{array}\right].
\end{equation*}
It can be verified that the thresholding operator $S_\lambda$ is Newton differentiable \cite{GriesseLorenz:2008},
and one Newton derivative operator $G(x,d)$ of the operator $F$ is given by
\begin{equation*}
  G(x,d) = \left[\begin{array}{cc}
    \Psi^t\Psi & I\\
    I_\mathcal{I} & - I_\mathcal{A},
  \end{array} \right],
\end{equation*}
with the active set $\mathcal{A}=\{i: |x_i+d_i|>\lambda\}$ and inactive set $\mathcal{I}=\mathcal{A}^c$.
Then, upon introducing the notation  $\calA_{k+1}^+=\{i: x_i^k+d_i^k>\lambda\}$, $\calA_{k+1}^-=\{i:x_i^k
+d_i^k<-\lambda\}$, $\calA_{k+1}=\calA_{k+1}^+\cup \calA_{k+1}^-$, and $\calI_{k+1}=\calA_{k+1}^c$,
the Newton update \eqref{eqn:Newton} is given by
\begin{align*}
  \left[\begin{array}{c}
    x^{k+1}\\ d^{k+1}\end{array}\right] & = \left[\begin{array}{c}x^k\\ d^k\end{array}\right] - G(x^k,d^k)^{-1}F(x^k,d^k),
\end{align*}
which, upon multiplying both sides by $G(x^k,d^k)$, can be recast into
\begin{align}
\Psi^t \Psi x^{k+1} + d^{k+1}  & =  0, \label{eqn:newton1}\\
I_{\mathcal{I}_{k+1}} (x^{k+1} - x^k) - I_{\mathcal{A}_{k+1}}(d^{k+1} - d^k) & = S_\lambda(x^k + d^k) - x^k. \label{eqn:newton2}
\end{align}
Meanwhile, by the definition of the soft-thresholding operator $S_\lambda$, we have
\begin{equation*}
S_\lambda(x_i^k + d_i^k) = \left\{\begin{array}{cc}
x^k_i + d^k_i -\lambda & i\in \calA_{k+1}^+, \\
0  & i\in \calI_{k+1},  \\
x^k_i + d^k_i +\lambda & i\in \calA_{k+1}^-. \\
\end{array}
\right.
\end{equation*}
Then equation \eqref{eqn:newton2} simplifies to
\begin{align*}
  x_{\mathcal{I}_{k+1}}^{k+1} &= \mathbf{0}_{\mathcal{I}_{k+1}}\quad\mbox{and}\quad
  d_{\calA_{k+1}}^{k+1} = \lambda[\mathbf{1}_{\calA_{k+1}^+}^t,-\mathbf{1}^t_{\calA_{k+1}^-}]^t.
\end{align*}
Upon substituting these identities into equation \eqref{eqn:newton1}, the semismooth Newton method gives rises to the following PDAS iteration:
\begin{align*}
  x_{\mathcal{I}_{k+1}}^{k+1} &= \mathbf{0}_{\mathcal{I}_{k+1}},\\
  d_{\calA_{k+1}}^{k+1} & = \lambda[\mathbf{1}_{\calA_{k+1}^+}^t,-\mathbf{1}^t_{\calA_{k+1}^-}]^t,\\
  \Psi^t_{\calA_{k+1}}\Psi_{\calA_{k+1}} x_{\calA_{k+1}}^{k+1} & = \Psi_{\calA_{k+1}}^ty - d_{\calA_{k+1}}^{k+1},\\
  d_{\calI_{k+1}}^{k+1} & =\Psi_{\calI_{k+1}}^ty - \Psi_{\calI_{k+1}}^t\Psi_{\calA_{k+1}}x_{\calA_{k+1}}^{k+1}.
\end{align*}

Thus, for the lasso problem \eqref{reglasso}, the semismooth Newton method can be
reformulated into a primal-dual active set (PDAS) algorithm. Due to the local
superlinear convergence of the semismooth Newton method, it is very efficient,
especially when coupled with a continuation strategy \cite{FanJiaoLu:2014}.
Actually it merits one-step convergence under suitable conditions. This
section presents a unified framework for developing PDAS type methods for
nonconvex sparse recovery based on the model \eqref{eqn:model}, which maintains
the excellent local convergence property.

\subsection{PDAS algorithm for nonconvex sparse recovery}

There are two key ingredients in constructing a PDAS algorithm:
\begin{itemize}
\item[(i)] to characterize the active set $\mathcal{A}$
by $x^*$ and $d^*$; 
\item [(ii)] to derive an explicit expression for the dual variable $d^*$ on
$\calA$.
\end{itemize}
We crucially exploit the optimality condition \eqref{eqn:nec} of a coordinate-wise minimizer $x^*$
to obtain the requisite ingredients (i) and (ii). Recall that the active set $\mathcal{A}$ of $x^*$ defined in Section
\ref{sec:op} is its support, i.e., $\mathcal{A} = \{i: x_i^* \neq 0\}.$ To see (i), by Lemma
\ref{lem:necopt} and the property of the operator $S^\rho$ in Lemma \ref{lem:proximal}, one observes
\begin{itemize}
\item for capped-$\ell^1$, SCAD and MCP penalties, $|x_i^* + d_i^*| > T^* \Leftrightarrow x_i^* \neq 0$,
\item for $\ell^\tau$ penalty, $0\leq \tau <1$, $\left\{\begin{array}{l}|x_i^* + d_i^*| > T^* \Rightarrow x_i^* \neq 0,\\ |x_i^* + d_i^*| < T^* \Rightarrow x_i^* = 0, \\ |x_i^* + d_i^*| = T^* \Rightarrow x_i^* = 0 \mbox{ or } t^*. \end{array}\right.$
\end{itemize}
Hence, except the case $|x_i^* + d_i^*| = T^*$ for the $\ell^0$ and bridge penalty, the active set
$\mathcal{A}$ can be determined by using both primal and dual variables. Next we derive
explicitly the dual variable $d^*$ on the set $\calA$, i.e., (ii). Straightforward computations show the formulas in
Table \ref{tab:activeset}; see Appendix \ref{app:dual} for details. We summarize these discussions in the following
proposition, which form the basis for constructing the PDAS algorithm below.

\begin{proposition}\label{prop:pdasupdate}
Let $x^*$ and $d^*$ be a coordinate-wise minimizer and the respective
dual variable, $\calA$ be the active set, and let
\begin{equation*}{\small
  \widetilde\calA = \left\{\begin{array}{ll}
    \left\{i:\ |x_i^*+d_i^*|=T^*\right\},  & \ell^0,\ \mathrm{bridge},\\ [1.2ex]
    \left\{i:  |x_i^* + d_i^*| = \lambda(\tau + \tfrac{1}{2})\right\}, & \mathrm{capped}\mbox{-}\ell^1,\\ [1.2ex]
    \emptyset, &   \mathrm{SCAD}, \ \mathrm{MCP}.
  \end{array}\right.}
\end{equation*}
If the set $\widetilde\calA=\emptyset$, then
$(\mathrm{i})$ $\mathcal{A}$ can be characterized by $\{i: |x_i^* + d_i^*| > T^*\}$, and
$(\mathrm{ii})$ the dual variable $d^*$ on $\calA$ can be uniquely written as in Table \ref{tab:activeset}.
\end{proposition}

The set $\widetilde{\calA}$ is always empty for the SCAD and MCP. For the $\ell^0$,
bridge and capped-$\ell^1$ penalty, it is likely empty, which, however, cannot
be \textit{a priori} ensured.

Using Proposition \ref{prop:pdasupdate}, now we are ready to derive a unified PDAS algorithm. First,
note that on the active set $\mathcal{A}$, the dual variable $d^*$ has two equivalent expressions,
i.e., the defining relation
\begin{equation*}
  \begin{aligned}
    \Psi_\calA^t(y-\Psi_\calA x^*_\calA) & = d_\calA^*,
  \end{aligned}
\end{equation*}
and the expression $d_\calA^*=d_{\calA}(x^*,d^*)$ from Proposition \ref{prop:pdasupdate}(ii).
This is the starting point for the PDAS algorithm. Similar to the case of convex optimization
problems \cite{HintermullerItoKunisch:2003,FanJiaoLu:2014}, at each iteration, with $(x_k,d_k)$
being the current primal and dual variables, first we
approximate the active set $\calA$ and inactive set $\calI$ by $\calA_k$ and $\calI_k$
respectively defined by
\begin{equation*}
  \calA_k=\{i: |x_i^{k-1}+d_i^{k-1}|> T^*\}\quad\mbox{and}\quad \calI_k = \calA_k^c.
\end{equation*}
Then we update the primal variable $x^k$ on the active set $\calA_k$ by
\begin{equation}\label{eqn:update}
  \Psi_{\calA_k}^t( y - \Psi_{\calA_k} x^k_{\calA_k}) =  p_{\calA_k},
\end{equation}
where $p_{\calA_k}$ is a suitable approximation of the dual variable $d^*$ on the active set
$\calA_k$ to be given below,  and set $x^k$ to zero on the inactive set $\calI_k$. Finally
we update the dual variable $d^k$ by
\begin{equation*}
  d^k = \Psi^t(y-\Psi x^k).
\end{equation*}
We summarize the above description in Algorithm \ref{alg:pdas}. It is important to note that the
algorithm takes a uniform form for all five nonconvex penalties, and the implementation is
straightforward and varies very little for different penalties: the only differences lie in
the value of $T^*$ and the approximate dual $p_{\calA_k}$. Note that Algorithm \ref{alg:pdas}
is a Newton type method, and good initial guess is required for the convergence. Clearly, an
inadvertent choice can seriously compromise the accuracy of the estimate. The important
issue of initial guess will be addressed below in Section \ref{ssec:cont}.

\begin{algorithm}[H]
   \caption{Unified primal-dual active set algorithm: $x_{\lambda} \leftarrow\textit{updas}(\rho,\tau,\lambda,K,x^0)$}\label{alg:pdas}
   \begin{algorithmic}[1]
     \STATE Input: Penalty $\rho$,  parameters $\tau$, $\lambda$, $K$. Set initial guess $x^0$ and find $d^0= \Psi^t( y -\Psi x^0)$.
     \FOR {$k=1,2,...K$}
     \STATE Compute the active and inactive sets $\calA_k$ and $\calI_k$ respectively by\\
          \quad \quad  \quad $ \calA_k = \{i: |x^{k-1}_i+ d^{k-1}_i| > T^*\}\quad\mbox{and}\quad \calI_k = \mathcal{A}_k^c,$ where, $T^*$ is given in Lemma  \ref{thm:tT}.
     \STATE Update the primal and dual variable $x^k$ and $d^k$ respectively by
       \begin{equation*}
        \left\{
          \begin{array}{l}
           x_{{\cal I}_k}^{k} = \textbf{0}_{{\cal I}_k}, \\[1.2ex]
           \Psi_{{\cal A}_k}^t \Psi_{{\cal A}_k} x_{{\cal A}_k}^{k} = \Psi_{{\cal A}_k}^t y - p_{\mathcal{A}_k},  \\[1.2ex]
           d^{k} = \Psi^t(\Psi x^k - y),
          \end{array}\right.
       \end{equation*}
       where $p_{\calA_k}$ is given in Table \ref{tab:activeset}.
     \STATE Check the stopping criterion.
     \ENDFOR
     \STATE Output: $x_{\lambda}$.
   \end{algorithmic}
\end{algorithm}

The choice of the approximate dual variable $p_{\calA_k}$ is related to the expression of the dual variable $d^*_\calA$,
cf. Proposition \ref{prop:pdasupdate}. For example, a natural choice of $p_{\calA_k}$ for
the bridge penalty is given by $p_i = \lambda\tau |x_i^{k}|^\tau/x_i^{k}$ for $i\in\calA_k$.
However, it leads to a nonlinear system for updating $x^k$.
In Algorithm \ref{alg:pdas} we choose an explicit expression for $p_{{\cal A}_k}$, cf. Table \ref{tab:activeset}, which
amounts to the one-step fixed-point iteration of the nonlinear equation. It is worth noting that this choice of $p_{{\cal A}_k}$
ensures its boundedness. That is, each component $p_{i}$ satisfies
\begin{equation}\label{eqn:bound}
|p_{i}| \leq \left\{\begin{array}{ll}
 0 & \ell^0,\\[1.2ex]
 \lambda^{\frac{1}{2-\tau}} (2(1-\tau))^{\frac{\tau-1}{2-\tau}} & \mathrm{bridge},\\[1.1ex]
\lambda & \mathrm{capped-}\ell^1, \mathrm{MCP} \\[1.1ex]
\frac{\tau}{\tau-1} \lambda & \mathrm{SCAD}.
\end{array}\right.
\end{equation}

\begin{table}[H]
  \centering
  \caption{Explicit expression of the dual variable $d^*_\calA$ on the active
   set $\calA=\{i: x^*_i\neq 0\}$, and its approximation $p_{\calA_k}$ on $\calA_k
   =\{i: |s_i^{k-1}|> T^*\}$, with $s_i=d_i+x_i$.}\label{tab:activeset}
   \vspace{-.3cm}{\scalebox{0.9}{
  \begin{tabular}{ll}
   \hline
   penalty          & $d_\calA^*$ \\
   \hline\\
    $\ell^0$        & 0                                        \\
    $\ell^\tau$     & $\lambda\tau \frac{|x_i^*|^\tau}{x_i^*}$ \\
    capped-$\ell^1$ & $\left\{\begin{array}{ll}
       0 & \mbox{if } |s_i^*| > \lambda(\tau + \tfrac{1}{2})\\[1.1ex]
      \mathrm{sgn}(s_i^*) \lambda & \mbox{if } \lambda < |s_i^*| < \lambda(\tau + \tfrac{1}{2})\\[1.1ex]
      \{0,\mathrm{sgn}(s_i^*) \lambda\} & \mbox{if } |s_i^*| = \lambda(\tau + \tfrac{1}{2})
      \end{array}\right.$ \\
    SCAD & $ \left\{\begin{array}{ll}
      0 & \mbox{if }|s_i^*|\geq \lambda\tau\\[1.2ex]
      \tfrac{1}{\tau-1}(\sgn(s_i^*)\lambda\tau - x_i^*) & \mbox{if }\lambda\tau >|s_i^*| > 2\lambda\\[1.1ex]
      \sgn(s_i^*)\lambda & \mbox{if } 2\lambda \geq |s_i^*| > \lambda\end{array}\right.$\\
    MCP & $\left\{\begin{array}{ll}0 & \mbox{if } |s_i^*| \geq \lambda\tau\\[1.1ex]
        \tfrac{1}{\tau}(\mathrm{sgn}(s_i^*)\lambda\tau - x_i^*) & \mbox{if }\lambda <|s_i^*| < \lambda\tau\end{array}\right.$
        \\
  \hline
   penalty         & $p_{\calA_k}$ \\
   \hline\\
    $\ell^0$       &          0                              \\
    $\ell^\tau$    & $\left\{\begin{array}{ll}               0 & |x_i^{k-1}|< t^* \\[1.1ex]
                                                             \lambda\tau \frac{|x_i^{k-1}|^\tau}{x_i^{k-1}} & |x_i^{k-1}|\geq t^* \end{array}\right.$, $t^* = (2\lambda(1-\tau))^{\frac{1}{2-\tau}}$\\
    capped-$\ell^1$ &$\left\{\begin{array}{ll}
     0 & \mbox{if } |s_i^{k-1}|\geq \lambda(\tau + \tfrac{1}{2})\\[1.1ex]
    \sgn(s_i^{k-1})\lambda & \mbox{if }\lambda <|s_i^{k-1}| <\lambda(\tau + \tfrac{1}{2}) \end{array}\right.$\\
    SCAD      &$\left\{\begin{array}{ll}
    \frac{1}{\tau-1}(\sgn(s_i^{k-1})\lambda\tau - x_i^{k-1})&\mbox{if }\lambda\tau > |s_i^{k-1}| > 2\lambda \mbox{ and } x_i^{k-1}\cdot d_i^{k-1} \geq 0\\[1.1ex]
    \mathrm{sgn}(s^{k-1}_i)\lambda &\mbox{if }2\lambda \geq |s_i^{k-1}| > \lambda\\[1.1ex]
    0 & \mbox{otherwise}
     \end{array}\right.$\\\\
    MCP &$\left\{\begin{array}{ll}
    \frac{1}{\tau}(\mathrm{sgn}(s^{k-1}_i)\lambda\tau - x_i^{k-1}) &\mbox{if }\lambda < |s_i^{k-1}| < \lambda\tau \mbox{ and } x_i^{k-1}\cdot d_i^{k-1} \geq 0\\[1.2ex]
  0&\mbox{otherwise }
   \end{array}\right.$\\
  \hline
  \end{tabular}}}
\end{table}

The stopping criterion at step 5 of Algorithm \ref{alg:pdas} is chosen to be either $\mathcal{A}_k = \mathcal{A}_{k+1}$  or $k\geq K $ for some fixed small integer  $K >0$.

\subsection{Continuation strategy and tuning parameter selection}\label{ssec:cont}

To successfully apply Algorithm \ref{alg:pdas} (i.e., UPDAS) to the model \eqref{eqn:model}, there are two important practical
issues, i.e., the initial guess $x^0$ in Algorithm \ref{alg:pdas} and the choice of the regularization parameter
$\lambda$, which we discuss separately below.

Since the PDAS algorithm is a Newton type method, it merits the highly
desirable fast (or superlinear) convergence, but only in the neighborhood of a minimizer. This is also expected to
be the case for the model \eqref{eqn:model}, in light of the nonconvexity of the penalties. Hence, in order to fully
exploit the fast local convergence feature, a good
initial guess is required, which unfortunately is often unavailable in practice. In this work, we adopt a
continuation strategy to arrive at a good initial guess, which serves the role of globalizing the PDAS algorithm.
Specifically, let $\lambda_s = \lambda_0 \gamma^s$, $\gamma\in(0,1)$, be a decreasing sequence of regularization parameters. Then
we apply Algorithm \ref{alg:pdas} on the sequence $\{\lambda_s\}_{s}$, with the solution $x_{\lambda_s}$
being the initial guess for the $\lambda_{s+1}$-problem. The overall algorithm is given in Algorithm \ref{alg:pdasc},
and termed as UPDASC.

\begin{algorithm}[H]
   \caption{Unified primal-dual active set with continuation algorithm (UPDASC)}\label{alg:pdasc}
   \begin{algorithmic}[1]
     \STATE Input $\lambda_0$ by \eqref{eqn:lambda0} and $\gamma \in (0,1)$. Let $x_{\lambda_0} = 0$.
     \FOR {$s=1,2,...$}
     \STATE   Run Algorithm  \ref{alg:pdas}  $\textit{updas}(\rho,\tau,\lambda,K,x^0)$ to problem \eqref{eqn:model} with $\lambda = \lambda_s :=\gamma^s\lambda_0$,  $x^0  = x_{\lambda_{s-1}}$ to get $x_{\lambda_{s}}$.
     \STATE If the optional stopping condition  $\|x_{\lambda_s}\|_0 > \lfloor \frac{n}{\log n} \rfloor$ holds,  stop.
     \ENDFOR
     \STATE Output: Solution path  $\{x_{\lambda_s}\}_{s= 1,2...}.$
   \end{algorithmic}
\end{algorithm}

The initial guess $\lambda_0$ is chosen large enough such that $0$ is the global
minimizer of the model \eqref{eqn:model}. In particular, we can choose it by
\begin{equation}\label{eqn:lambda0}
\lambda_{0} = \left\{\begin{array}{ll}
 \frac{1}{2}\|\Psi^t y\|_{\infty}^2 & \ell^0,\\[1.1ex]
 (\frac{\|\Psi^t y\|_{\infty}}{2-\tau})^{2-\tau} (2(1-\tau))^{1-\tau} & \mbox{bridge},\\[1.1ex]
\|\Psi^t y\|_{\infty} & \mbox{capped-}\ell^1, \mbox{SCAD, MCP}.
\end{array}\right.
\end{equation}
It is worth noting that with this choice of $\lambda_0$, Algorithm \ref{alg:pdasc}
is essentially free from the troublesome issue of choosing initial guess.

The regularization parameter $\lambda$ in the model \eqref{eqn:model}  compromises the tradeoff between
the data fidelity and the sparsity level of the solution, and it plays a crucial role in obtaining
good reconstructions. However, how to choose a proper value in high-dimension is one notoriously challenging
problem. There are several possible rules, e.g., cross validation, balancing principle \cite{ItoJin:2015}, L-curve, Bayesian
information criterion \cite{WangKimLi:2013}.
In this work, we advocate the following simple approach: first run Algorithm \ref{alg:pdasc} (i.e., UPDASC) to
obtain a solution path until, e.g., $\|x_{\lambda_s}\|_0 > \lfloor \frac{n}{\log n} \rfloor$
for some $s$, say, $s = S$. Let $\Lambda_{\ell} = \{\lambda_s: \|x_{\lambda_s}\|_0  = \ell, s = 1,..., S\},
\quad  \ell  =1,...,\lfloor \frac{n}{\log n} \rfloor$ be the set of tuning parameter at which the output of
UPDAS has $\ell$ nonzero elements. Then we determine the optimal $\lambda$ by voting \cite{HuangJiaoLuZhu:2017}, i.e.,
\begin{equation}\label{autoreg}
  \hat{\lambda} = \max \{\Lambda_{\bar{\ell}}\} \quad  \textrm{and} \quad \bar{\ell } = \arg\max_{\ell}\{|\Lambda_{\ell}|\}.
\end{equation}
The tuning parameter selection  rule \eqref{autoreg} is seamlessly integrated with the
continuation strategy without any extra computational overhead, since the requisite solutions
along the path have been all obtained by UPDASC algorithm. In practice, the approach works strikingly
well; see example \ref{exam:recovery} in Section \ref{sec:num} for an illustration.

\subsection{Consistency of UPDASC}
Last we discuss the consistency  of Algorithm \ref{alg:pdasc}. We shall focus on noise free data, and
in the presence of noise, one can similarly  derive a bound that is proportion to noise level  on the
estimation  error, i.e., the error between the output and the underlying regression target, but the proof is much more involved (see
\cite{HuangJiaoLuYang:2016} and \cite{JiaoJinLu:2015}  for the lasso and $\ell^0$ cases, respectively).
Let $y = \Psi x^\dag$, where $x^\dag$ is the target  sparse vector with its active set ${\cal A}^\dag = \{i:
x_i^\dag \neq 0\}$ and $T = |{\cal A}^\dag|$. The restricted isometry property (RIP) \cite{CandesRombergTao:2006} of
order $k$ with constant $\delta_k$ of a matrix $\Psi$ is defined as follows: Let $\delta_k \in (0, 1)$ be the smallest constant such that
\begin{equation*}
(1-\delta_k)\|x\|^2 \leq \|\Psi x\|^2 \leq (1+\delta_k)\|x\|^2
\end{equation*}
holds for all $x$ with $\|x\|_0 \leq k$. Now we make the following assumption.
\begin{assumption}\label{assumption:1}
The matrix $\Psi$ satisfies the RIP condition with a RIP constant
$$\delta\equiv\delta_{T+1}\leq \left\{\begin{array}{ll}\frac{1}{\sqrt{5T}+1} & \mathrm{capped-}\ell^1, \mathrm{ MCP}, \\[1.2ex]
\frac{1}{\sqrt{8T} +1} & \mathrm{SCAD}, \\[1.2ex]
\frac{2-\tau}{2-\tau + \sqrt{T[(4- 2\tau)^2 +1]}} & \mathrm{bridge}. \end{array}\right.$$
\end{assumption}

\begin{remark}
In the absence of regularity conditions on design matrix $\Psi$, target
solution $x^{\dag}$ and starting values, the active set sequence generated via PDAS
algorithm may cycle (see, e.g., \cite{Han:2015,JiaoJinLu:2015}), and thus the convergence
of the inner iterate can generally not be guaranteed. However, the convergence of
UPDASC, i.e., with continuation, is ensured, provided certain conditions, e.g., on
the matrix $\Psi$, are satisfied.
\end{remark}

Now we can state the convergence of Algorithm \ref{alg:pdasc} under Assumption \ref{assumption:1}, and the
lengthy and technical proof is deferred to Appendix \ref{app:convergence}. The well-definedness means that
the linear system for updating the primal variable is invertible. Note also that the inner iteration can
always terminate, due to the choice of a finite maximum number of iterations.
\begin{theorem}\label{thm:convergence}
Let Assumption \ref{assumption:1} hold. Then for any $\gamma\in(0,1)$ sufficiently close to $1$ {\rm(}the precise
range of $\gamma$ is given explicitly in the proof of Theorem \ref{thm:convergence} in Appendix
\ref{app:convergence}{\rm)}  Algorithm \ref{alg:pdasc} is well-defined and
\begin{equation*}
\|x_{\lambda_s} - x^\dag\|\leq \frac{\sqrt{T}C_{\lambda_s}}{1-\delta}
\end{equation*}
as $s \geq \mathcal{O} (\log_{\frac{1}{\gamma}} \frac{\lambda_0}{|x_i^\dag|_{\min}})$, where
$|x^\dag|_{\min} =  \min\left\{|x_i^\dag|: x_i^\dag \neq 0 \right\}$ and
\begin{equation*}
C_{\lambda_s}= \left\{\begin{array}{ll}
 0 & \ell^0,\\[1.2ex]
 \lambda_{s}^{\frac{1}{2-\tau}} (2(1-\tau))^{\frac{\tau-1}{2-\tau}} & \mathrm{bridge},\\[1.1ex]
\lambda_s & \mathrm{capped-}\ell^1, \mathrm{MCP} \\[1.1ex]
\frac{\tau}{\tau-1} \lambda_s & \mathrm{SCAD}.
\end{array}\right.
\end{equation*}
\end{theorem}

Theorem \ref{thm:convergence} implies that for noise free data, the UPDASC solution is consistent with the true
sparse solution $x^\dag$ when $s$ is large enough for the  $\ell^0$ regularized model and as $s\rightarrow +\infty$
for other nonconvex models, respectively. Further, the proof of Theorem \ref{thm:convergence} in
Appendix \ref{app:convergence} indicates that the continuation strategy actually allows a precise control
over the evolution of the active set during the iteration, cf. Lemma \ref{lem:monotone}, in addition
to providing a good initial guess. These observations clearly show the viability of the continuation strategy as
a globalization technique for the PDAS algorithm for nonconvex sparse recovery models.

\begin{remark}
The recovery guarantee in Theorem \ref{thm:convergence} relies on RIP type conditions, and similar conditions
were used for orthogonal matching pursuit \cite{HuangZhu:2011,MoShen:2012}. Note that for lasso, RIP
type condition of the form that $\delta_{aT}$ for some $a>1$ is sufficiently small on the matrix $\Psi$ ensures
stable recovery \cite{CandesRombergTao:2006}. Further, the restricted eigenvalue condition and minimal signal
strength condition provide statistical guarantee for the global minimizers for a class of nonconvex models
\cite{ZhangZhang:2012}; see also \cite{FengZhang:2017} for more recent refinements. It is enormous interest to
derive performance guarantee for Algorithm \ref{alg:pdasc} under analogous
conditions to these alternatives for lasso, thereby filling the gap between the theory and extremely encouraging
empirical success. One challenge of such an analysis for UPDASC is to bound the estimation error dynamically.
\end{remark}

\section{Numerical experiments and discussions} \label{sec:num}
In this section we showcase the performance of Algorithm \ref{alg:pdasc} (UPDASC) for the nonconvex penalties
in Table \ref{tab:nonconvex} on both simulated and real data. All the experiments are done on a four core
desktop with 3.47 GHz and 8 GB RAM. The \texttt{MATLAB} and \texttt{R} packages (Unified-PDASC) are available
at the following links \url{http://www0.cs.ucl.ac.uk/staff/b.jin/software/updasc.zip} and
\url{https://github.com/gordonliu810822/PDAS}, respectively.

\subsection{Experiment setup}
First we describe the problem setup, i.e., data generation and parameter choice. In all numerical examples except Example \ref{exam:GWAS},
the underlying true target  $x^{\dag}$ is given, and the response vector $y$ is generated by $y = \Psi x^{\dag} + \eta$,  where,  $\eta$  $\sim \mathcal{N}(0,\sigma^2 I_n)$ denotes the  noise.
 Unless otherwise stated, the standard deviation $\sigma$ is fixed at $\sigma=0.5$.


The matrix
$\Psi$ is generated as follows.
\begin{itemize}
\item[(i)] The rows of $\Psi$ are iid samples  from  $\mathcal{N}(0,\Sigma)$ with $\Sigma_{k,\ell} = \mu^{|k-\ell|}, 1 \le k, \ell \le p$, $\mu \in (0,1)$, we keep the convention $0^{0} = 1$. Unless otherwise stated, we set $\mu = 0.5$.
\item[(ii)] Random Gaussian matrix of size $n\times p$ with auto-correlation. First we
generate a random Gaussian matrix $\widetilde{\Psi} \in\mathbb{R}^{n\times p}$ with its
entries following i.i.d. $\mathcal{N}(0,1)$. Then we define a matrix
$\Psi\in\mathbb{R}^{n\times p}$ by setting $\psi_1 = \widetilde{\psi}_1$,
\begin{equation*}
   \psi_j = \widetilde{\psi}_j + 0.2*(\widetilde{\psi}_{j-1} + \widetilde{\psi}_{j+1}), \ \ j=2,...,p-1,
\end{equation*}
and $\psi_p= \widetilde{\psi}_p$.
\end{itemize}
These matrices are then normalized to have unit column norm.

The target  $x^\dag$ is a $T$-sparse vector whose support is uniformly distributed in $\{1,2,...p\}$
with  $\max \{|x^{\dag}_i|: x^{\dag}_i \neq 0\} = M$, $\min \{|x_i^{\dag}|: x^{\dag}_i \neq 0\} = m.$
Below we set $M = 10, m = 1$. We set $\lambda_{\max} = \lambda_0$ as in
\eqref{eqn:lambda0} and $\lambda_{\min} = 10^{-8} \lambda_{\max}$. The interval $[\lambda_{\min},
\lambda_{\max}]$ is divided into $N$ ($N=100$ in this work) equal subintervals
on a log-scale and let $\lambda_s$, $s=0,...,N$, be the $s$-th value (in descending
order). Unless otherwise specified, we set $\tau = 0.5,\
3.7,\ 2.7$, and 1.5 for the bridge, SCAD, MCP and capped-$\ell^1$ penalty, respectively.
Note the value of the parameter $\tau$ influences the convergence behavior of the UPDAS
algorithm and statistical estimates. However, an in-depth study of the interesting
theoretical question is beyond the scope of this work.

\subsection{Numerical results and discussions}
Now we present numerical examples to illustrate the
accuracy and efficiency of  Algorithm \ref{alg:pdasc} with tuning parameter selection rule \eqref{autoreg}.

\begin{exam}\label{exam:recovery}
The first test illustrates the accuracy of the proposed tuning parameter selection rule  \eqref{autoreg}.
We compute the exact support recovery probability  at different sparsity levels for all the penalties in
Table \ref{tab:nonconvex} (including lasso), where, $\hat{\lambda}$ is determined by \eqref{autoreg}.
The matrix $\Psi\in\mathbb{R}^{500\times 1000}$ is generated according to setting (i), the true signal $x^{\dag}$
has a support size of $10:10:150$, that is, from 10 to 150 with a step size of 10, and the noise standard
deviation is $\sigma = 0.1$.
\end{exam}

The ratio $\frac{\sum _{i = 1}^{100}\textbf{1}_{\textrm{supp}(x_{\hat{\lambda}}) = \textrm{supp}(x^{\dag})}}{100}$
is computed from 100 independent realizations. It is observed from Fig. \ref{fig:sparsitylevelexact} that the
proposed tuning parameter selection rule \eqref{autoreg} can determine correct solutions for all five nonconvex
penalties as the sparsity level varies from $10$ to $100$. Thus, the rule \eqref{autoreg} represents a feasible
approach for selecting the nontrivial parameter $\lambda$. From the figure, the superiority of the nonconvex
approaches for support detection over the convex approach is clearly observed.

\begin{figure}[H]
\centering
\includegraphics[trim = 0cm 0cm 0cm 0cm, clip=true,width = 10cm]{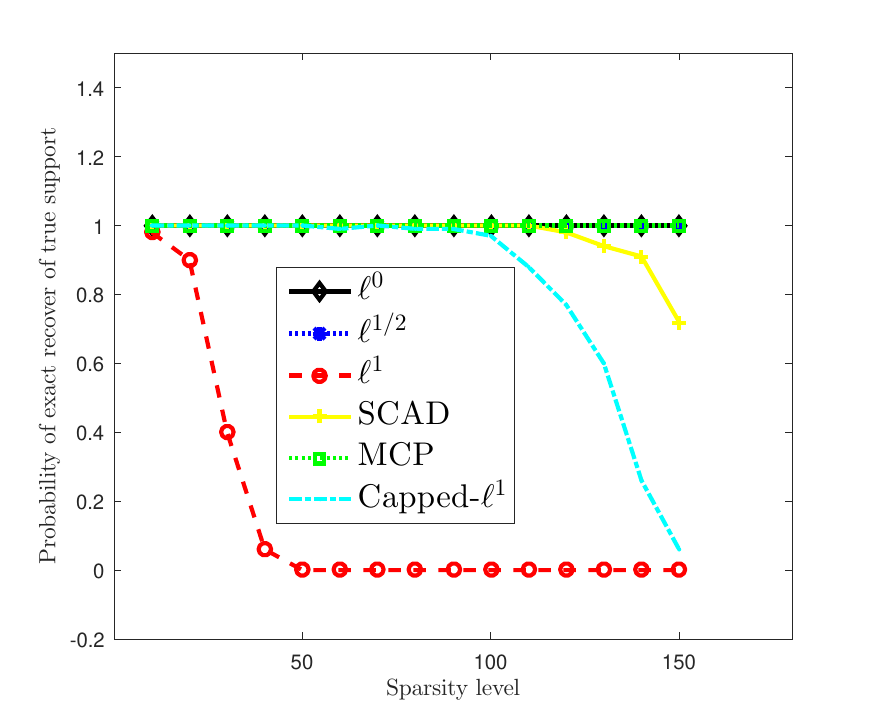}
\caption{The exact support recovery probability of the solution selected  by  \eqref{autoreg}   for Example \ref{exam:recovery}.
}\label{fig:sparsitylevelexact}
\end{figure}

In the next experiment, we compare Algorithm \ref{alg:pdasc} with an existing general iterative
shrinkage and thresholding algorithm (GIST) \cite{GongZhangLuHuangYe:2013} (available online at
\url{http://www.public.asu.edu/~jye02/Software/GIST/}). We also compare its efficiency with GLMNET, one of the
fastest lasso solvers currently available. We run GIST and GLMNET along the same path that is used in UPDASC
with the tuning parameter selection rule \eqref{autoreg}.
\begin{exam}\label{exam:perfcomparison}
We consider the following three different problem settings:
\begin{itemize}
  \item[$(\mathrm{a})$] The matrix $\Psi\in\mathbb{R}^{500\times5000}$ is generated according to setting (i),
    and the signal $x^{\dag}$ contains $20$ nonzero elements.
  \item[$(\mathrm{b})$] The matrix $\Psi\in\mathbb{R}^{1000\times 10000}$ is is generated according to setting (ii) and the signal $x^{\dag}$ contains
    $50$ nonzero elements.
  \item[$(\mathrm{c})$] The matrix $\Psi\in\mathbb{R}^{1000\times100000}$  is generated according to setting (i) with $\mu  = 0$, and the signal $x^{\dag}$
    contains $50$ nonzero elements.
\end{itemize}
\end{exam}

The performance is evaluated in terms of average CPU time (in seconds) and average relative error  defined as $\frac{\|x_{\hat{\lambda}} - x^{\dag}\|}{\|x^{\dag}\|}$,
which are computed based on  10 independent realizations of the problem setup. The
results are summarized in Tables \ref{updaste1}-\ref{updaste3}.

\begin{table}[H]
  \caption{Results for Example \ref{exam:perfcomparison}(a), CPU time in seconds and relative error (RE)} 
 \label{updaste1}
  \vspace{-.3cm}
  \begin{center}
  \begin{tabular}{ccccccccc}
\hline \hline
 &\multicolumn{2}{c}{UPDASC}
 &&\multicolumn{2}{c}{GIST}&&\multicolumn{2}{c}{GLMNET}\\
\cline{1-4}\cline{5-6} \cline{8-9}
                      &time       & RE      &&time    & RE       &&time      & RE       \\
 $\ell^{0}$           & {0.15}   & {4.70e-3}     && -      &-        && -    &-     \\
 $\ell^{1/2}$         & {0.13}    & {4.60e-3}     &&-       &-        && -   &-       \\
 SCAD                 & {0.10}    & {4.50e-3}    &&2.07    &{4.50e-3}  && -    &-       \\
 MCP                  & {0.09}   & {4.50e-3}   &&2.13    &{4.50e-3}   &&-    &-      \\
 capped-$\ell^{1}$    & {0.09}   & {4.50e-3}    &&1.38    &{4.50e-3}  && -    &-    \\
 lasso                & -         &  -   && -      &-       && 0.23    &3.76e-2      \\
\hline
  \hline
  \end{tabular}
  \end{center}
\end{table}

\begin{table}[H]
  \caption{Results for Example \ref{exam:perfcomparison}(b), CPU time in seconds and relative error (RE)} 
 \label{updaste2}
  \vspace{-.3cm}
  \begin{center}
  \begin{tabular}{ccccccccc}
\hline \hline
 &\multicolumn{2}{c}{UPDASC}
 &&\multicolumn{2}{c}{GIST}&&\multicolumn{2}{c}{GLMNET }\\
\cline{1-4}\cline{5-6} \cline{8-9}
                      &time       & RE      &&time    & RE       &&time      & RE       \\
 $\ell^{0}$           &  {0.60}    & {3.10e-3}     && -      &-        && -    &-    \\
 $\ell^{1/2}$         &  {0.42}    & {3.30e-3}     &&-       &-        && -   &-       \\
 SCAD                 &  {0.32}    & {3.10e-3}    &&7.75    &{3.10e-3} && -    &-       \\
 MCP                  &  {0.31}   & {3.10e-3}  &&7.80     &{3.10e-3}  &&-    &-      \\
 capped-$\ell^{1}$    &  {0.31}   & {3.10e-3}    &&4.67    &{3.10e-3} && -    &-      \\
 lasso                & -         &  -   && -      &-       && 1.06    &4.20e-2      \\
\hline
  \hline
  \end{tabular}
  \end{center}
\end{table}

\begin{table}[H]
  \caption{Results for Example \ref{exam:perfcomparison}(c), CPU time in seconds and relative error (RE)} 
 \label{updaste3}
  \vspace{-.3cm}
  \begin{center}
  \begin{tabular}{ccccccccc}
\hline \hline
 &\multicolumn{2}{c}{UPDASC}
 &&\multicolumn{2}{c}{GIST}&&\multicolumn{2}{c}{GLMNET }\\
\cline{1-4}\cline{5-6} \cline{8-9}
                      &time       & RE      &&time    & RE       &&time      & RE       \\
 $\ell^{0}$           &  {5.20}    & {3.40e-3}     && -      &-          && -    &-    \\
 $\ell^{1/2}$         &  {3.76}    & {3.50e-3}     &&-       &-          && -   &-       \\
 SCAD                 &  {2.60}    & {3.40e-3}    &&75.5    &{3.40e-3}   && -    &-       \\
 MCP                  &  {2.59}   &  {3.40e-3}  &&75.9     & {3.40e-3}  &&-    &-     \\
 capped-$\ell^{1}$    &  {2.60}   &  {3.40e-3}    &&13.6    & {4.65e-1} && -    &-   \\
 lasso                & -         &  -   && -      &-    && 12.4    &6.16e-2      \\
\hline
  \hline
  \end{tabular}
  \end{center}
\end{table}

Since GIST does not support $\ell^\tau, \tau \in [0,1)$, we do not present the corresponding results, indicated by $-$ in the tables.
For all three cases, the proposed UPDASC is much faster than GIST and GLMNET (on average by a factor of ten-thirty
and two-five when compared with GIST and GLMNET, respectively). While the reconstruction errors by the proposed UPDAS algorithm
is almost identical with that by GIST (except the capped-$\ell^1$  in  case (c), where, our result is 100 times smaller than
GIST) and about ten times smaller than that for GLMNET. This is attributed to its local superlinear convergence, which we
shall examine more closely below. These numerical results show clearly the huge potential of the proposed UPDASC algorithm for
nonconvex sparse recovery.

We now examine the continuation strategy and local supperlinear convergence of the algorithm.
\begin{exam}\label{exam:continuation}
The matrix $\Psi\in\mathbb{R}^{200\times1000}$ is generated according to (i), and the
signal $x^{\dag}$ contains $20$ nonzero elements  and $\sigma = 0.1$.
\end{exam}

The convergence history of Algorithm \ref{alg:pdasc} for Example \ref{exam:continuation} is shown in Fig. \ref{fig:convsup}.
In the figure, the notation $A$ and $A_s$ refer respectively to the exact active set $\textrm{supp}(x^{\dag})$ and the
approximate one $\textrm{supp} (x_{\lambda_{s}})$, where $x_{\lambda_{s}}$ is the solution to the $ \lambda_s$-problem. It is
observed that the size $|A_s|$ increases monotonically as the iteration proceeds. At each $\lambda_{s+1}$, with the solution
$x_{\lambda_{s}}$ as the initial guess, Algorithm \ref{alg:pdas} generally converges within three iterations for all five
nonconvex penalties, cf., Fig. \ref{fig:convsup}, which shows clearly the highly desirable local superlinear convergence of the
algorithm. Hence, the coverall procedure in Algorithm \ref{alg:pdasc} is very efficient.

\begin{figure}[H]
  \centering
  \begin{tabular}{cccc}
    \includegraphics[trim = 0cm 0cm 0cm 0cm, clip=true,width=3.4cm]{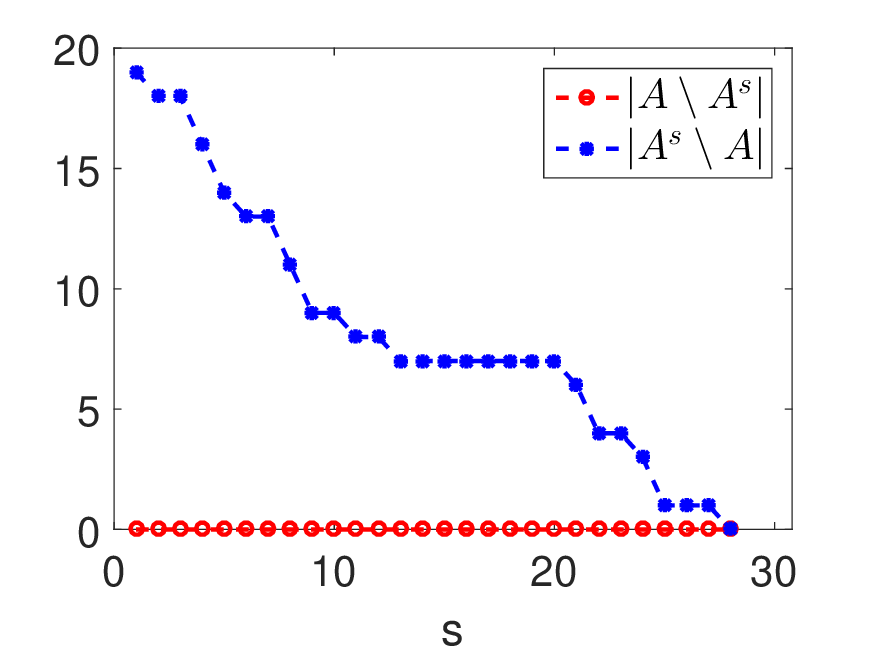} &\includegraphics[trim = 0cm 0cm 0cm 0cm, clip=true,width=3.4cm]{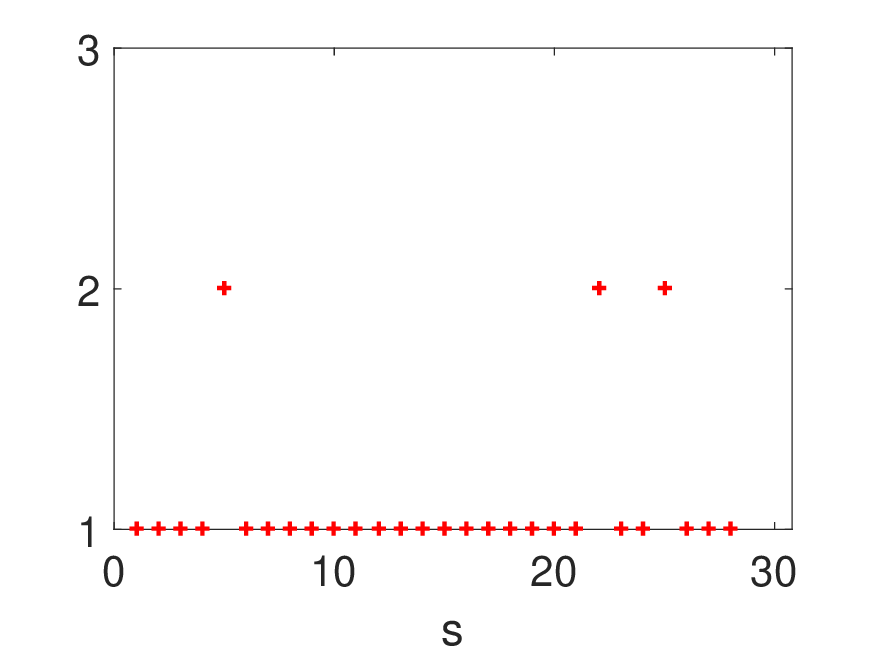} &
    \includegraphics[trim = 0cm 0cm 0cm 0cm, clip=true,width=3.4cm]{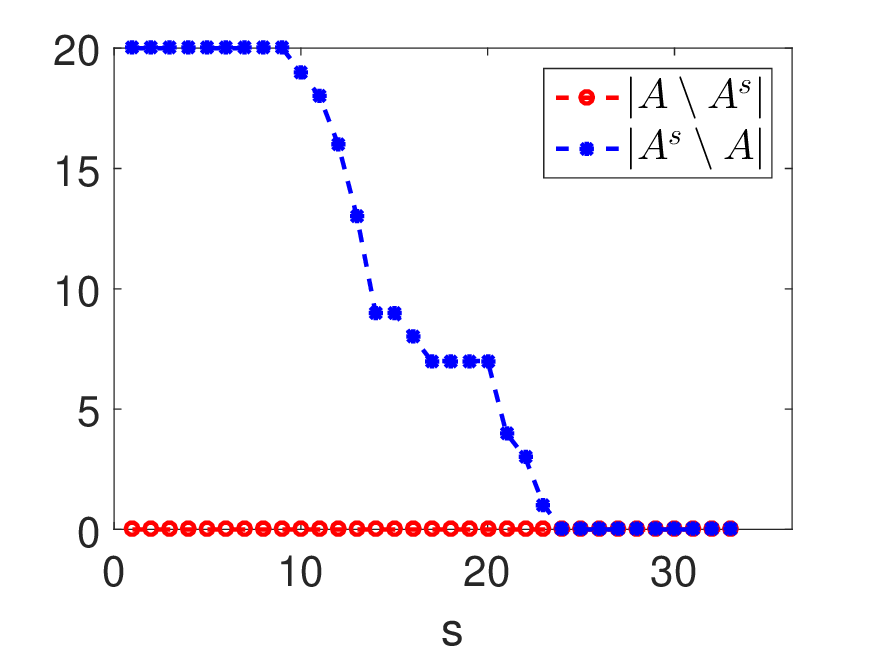} & \includegraphics[trim = 0cm 0cm 0cm 0cm, clip=true,width=3.4cm]{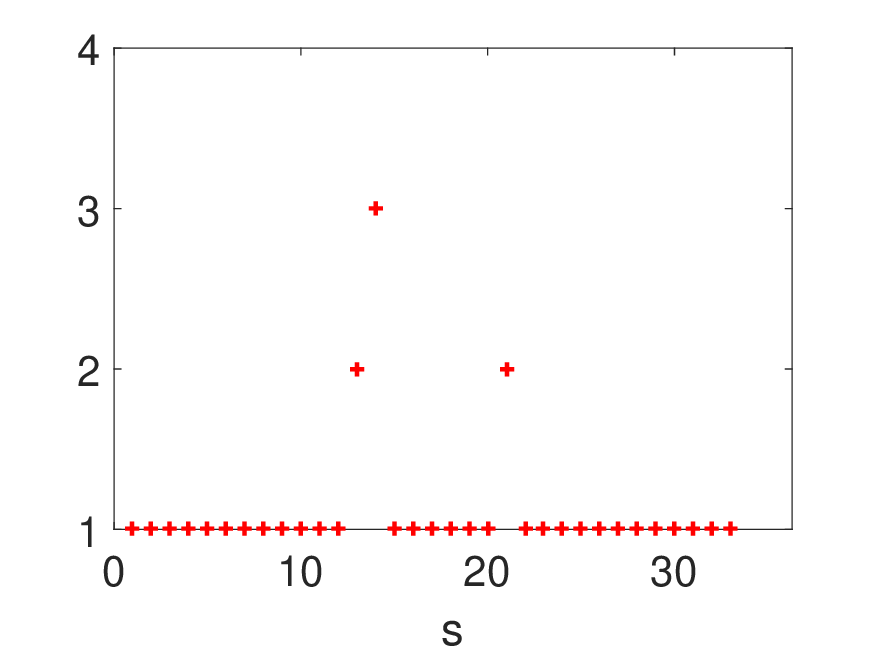}\\
    \multicolumn{2}{c}{(a) $\ell^0$} & \multicolumn{2}{c}{(b) $\ell^{1/2}$}\\
    \includegraphics[trim = 0cm 0cm 0cm 0cm, clip=true,width=3.4cm]{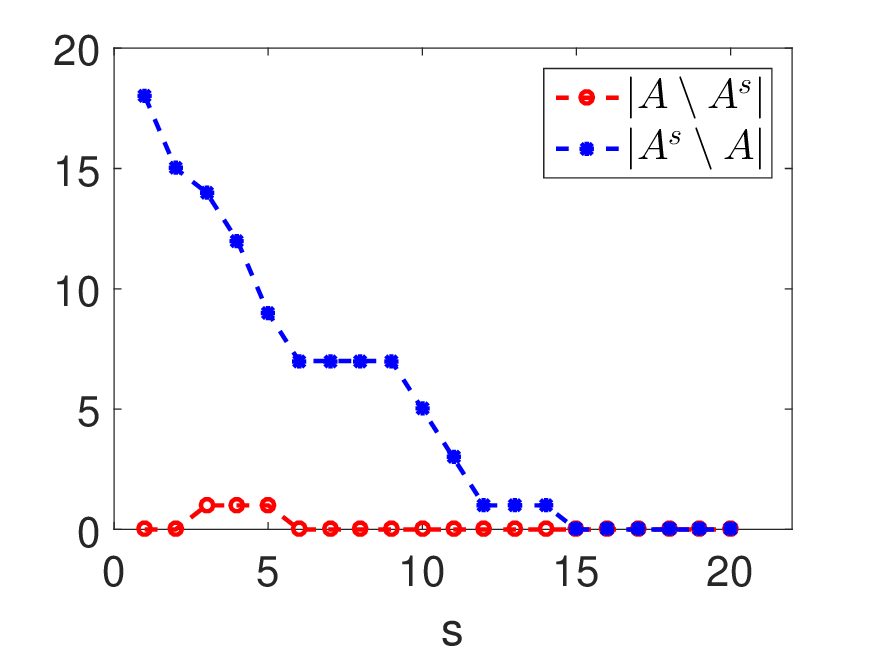} &\includegraphics[trim = 0cm 0cm 0cm 0cm, clip=true,width=3.4cm]{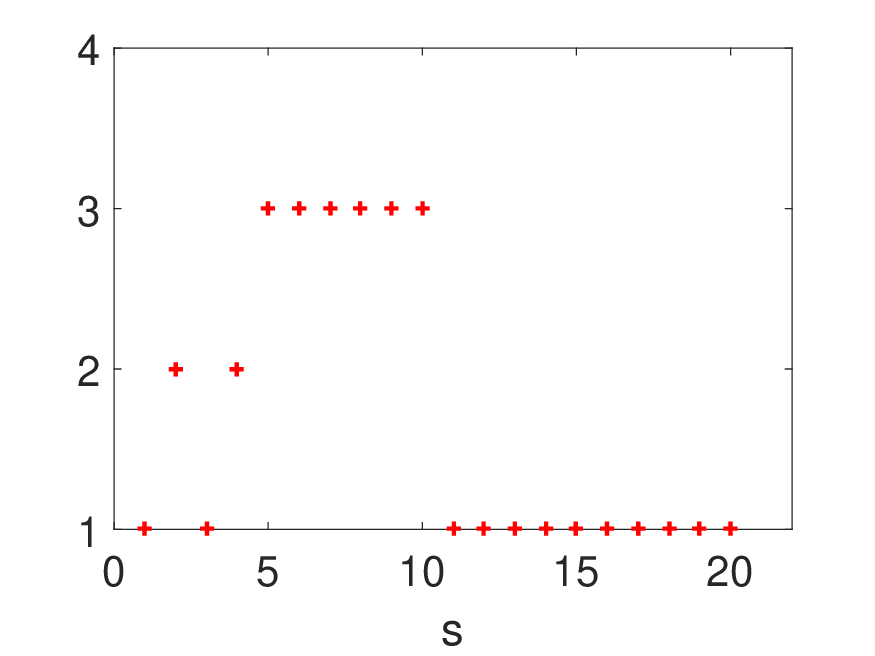} &
    \includegraphics[trim = 0cm 0cm 0cm 0cm, clip=true,width=3.4cm]{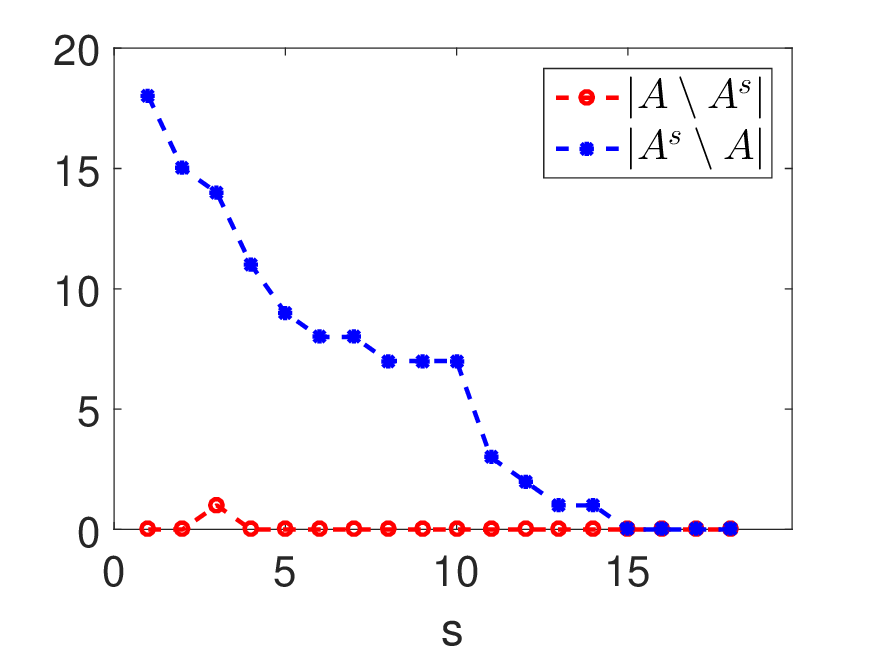} &\includegraphics[trim = 0cm 0cm 0cm 0cm, clip=true,width=3.4cm]{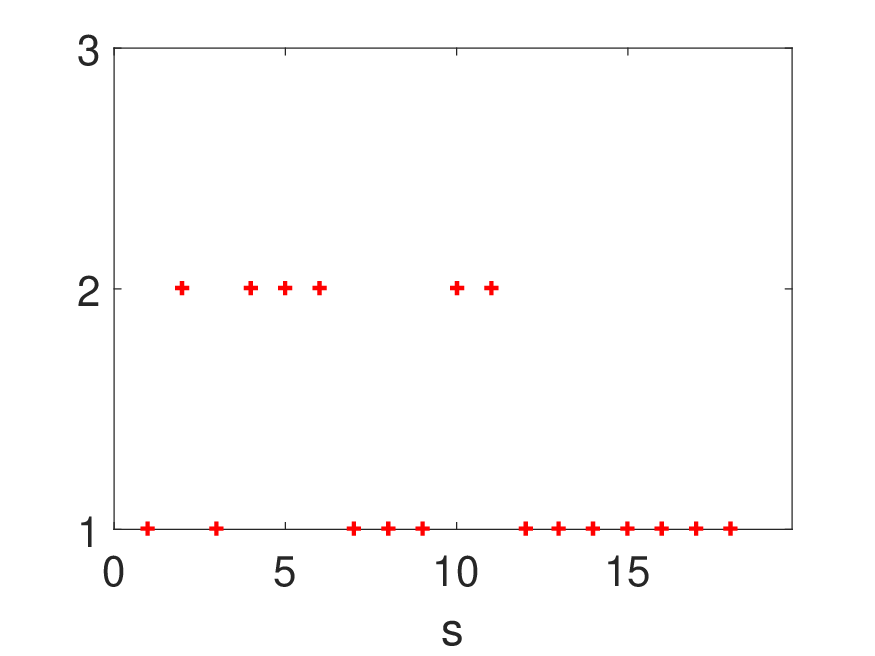}\\
     \multicolumn{2}{c}{(d) SCAD} & \multicolumn{2}{c}{(e) MCP}\\
    \includegraphics[trim = 0cm 0cm 0cm 0cm, clip=true,width=3.4cm]{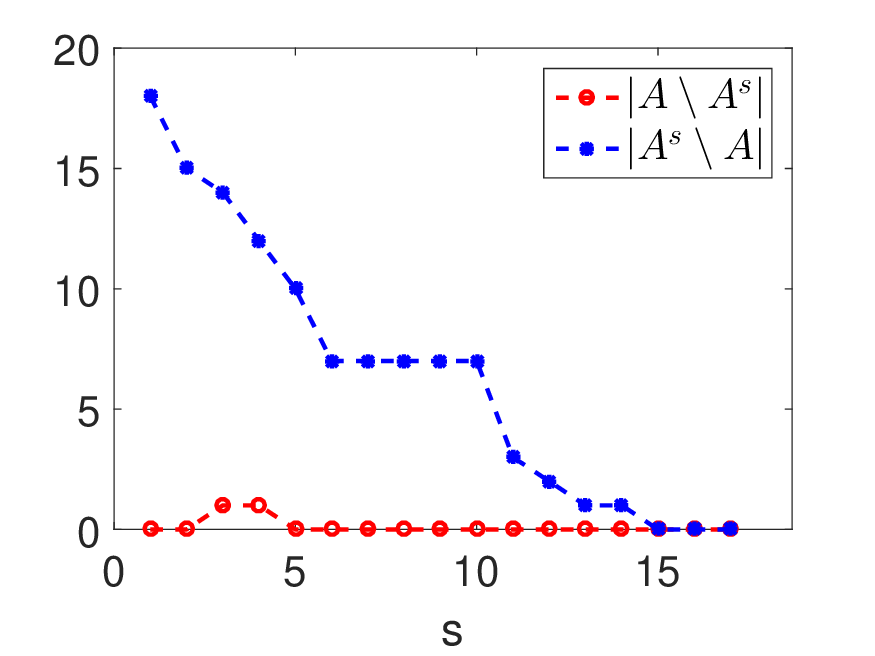} &\includegraphics[trim = 0cm 0cm 0cm 0cm, clip=true,width=3.4cm]{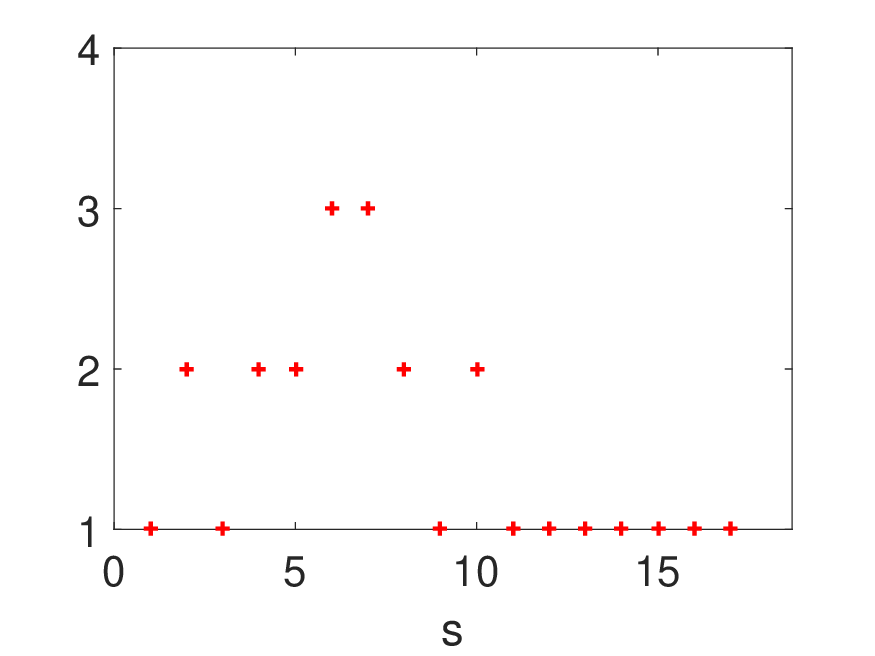} &\\
    \multicolumn{2}{c}{(c) capped-$\ell^1$} &
  \end{tabular}
  \caption{Convergence behavior of Algorithm \ref{alg:pdasc} for Example \ref{exam:continuation}: the variation of
  the active sets (left panel), and the number of iterations needed for each $\lambda_s$-problem (right panel).
  In the left panel, the vertical axis is the number of iterations, and in the right panel, the vertical axis is the
  size of the support. The horizontal axis in both panels is the index of the grid point in the solution path. }\label{fig:convsup}
\end{figure}

Four (bridge, capped $\ell^1$, SCAD and MCP) of the penalties in
Table \ref{tab:nonconvex} have a free parameter $\tau$ that controls
their concavity. Our next experiment examines the sensitivity of the
algorithm with respect to the concavity parameter $\tau$.

\begin{exam}\label{exam:sensitivity}
The matrix $\Psi\in \mathbb{R}^{400\times1000}$ is generated according to setting (i), and the true signal $x^{\dag}$
contains $10$ nonzero elements with  $M =100$.
\end{exam}

We evaluate Algorithm \ref{alg:pdasc} by CPU time (in seconds), relative error (in the $\ell^2$-norm) and
absolute $\ell^\infty$ error ($\|x_{\hat{\lambda}} - x^{\dag}\|_{\infty}$) computed from ten independent
realizations of the problem setup. The CPU time is fairly robust with respect to the concavity parameter
$\tau$, cf. Table \ref{tab:sensitivity}. Further, the reconstruction error varies little with the parameter
$\tau$, indicating the robustness of the penalized models. Interestingly, the reconstruction errors are
largely comparable across different penalties.

The consistency results in Theorem \ref{thm:convergence} employ the assumption that the matrix $\Psi$ satisfies the RIP assumption
\ref{assumption:1}, which in turn implies that there can only exist very weak correlations among the columns of
the matrix $\Psi$. The next experiment examines the robustness of UPDASC with respect to the correlation parameter $\mu $
in data generation setting (i).

\begin{exam}\label{exam:condition}
The matrix $\Psi\in \mathbb{R}^{300\times1000}$ is generated according to setting (i) with $\mu = 0.1:0.2:0.7$, and the true signal $x^{\dag}$
contains $10$ nonzero elements.
\end{exam}

\begin{table}[H]
  \caption{Results for Example \ref{exam:sensitivity}: sensitivity analysis.}
 \label{tab:sensitivity}
  \vspace{-0.3cm}
  \begin{center}
  \begin{tabular}{ccccccccc}
  \hline\hline
  \multicolumn{4}{c}{(a) bridge}
  &&\multicolumn{4}{c}{(b) capped-$\ell^1$}\\
  \cline{1-4}  \cline{6-9}
  $\tau$         &Time       & RE    & AE             &&  $\tau$   & Time       & RE    & AE  \\
   $0$           & 4.40e-2     &5.44e-4  &    4.62e-2 & &  $1.1$    &4.86e-2       & 5.01e-4       &    4.23e-2\\
   $0.2$         & 3.79e-2     &5.03e-4     & 4.26e-2 & &  $1.5$    &3.59e-2       & 5.06e-4       &    4.23e-2 \\
   $0.4$         & 3.47e-2     &5.54e-4     & 4.75e-2 & &  $5$      &3.70e-2       & 5.17e-4       &    4.23e-2  \\
   $0.6$         & 3.27e-2     &5.99e-4     & 4.85e-2 & &  $10$     &3.46e-2       & 5.25e-4       &    7.20e-2      \\
   $0.8$         & 2.88e-2     &8.44e-4     & 5.93e-2 & &  $30$     &3.58e-2       & 1.40e-3       &    1.11e-1   \\
  \hline
  \hline
  \multicolumn{4}{c}{(c) SCAD}
  &&\multicolumn{4}{c}{(d) MCP}\\
  \cline{1-4}  \cline{6-9}
   $\tau$        & Time       & RE    & AE       &&  $\tau$        & Time       & RE    & AE \\
   $2.1$         & 5.72e-2    & 5.03e-4 &  4.23e-2      &&  $1.1$         & 6.30e-2         &     5.03e-4       & 4.23e-2      \\
   $3.7$         & 4.70e-2    & 5.03e-4 &  4.23e-2      &&  $2.7$         & 5.07e-2         &     5.03e-4       & 4.23e-2     \\
   $5$           & 4.31e-2    & 5.03e-4 &  4.23e-2      &&  $5$           & 4.90e-2         &     5.03e-4       & 4.23e-2     \\
   $10$          & 4.33e-2    & 5.03e-4 &  4.23e-2      &&  $10$          & 4.87e-2         &     5.03e-4       & 4.24e-2       \\
   $30$          & 4.34e-2    & 8.66e-4 &  4.24e-2      &&  $30$          & 5.32e-2         &     8.87e-4       & 7.37e-2  \\
  \hline
  \hline
  \end{tabular}
  \end{center}
\end{table}

The consistency results in Theorem \ref{thm:convergence} rely on the assumption that $\Psi$ satisfies the RIP
i.e., Assumption \ref{assumption:1}, which in turn implies there can only be very weak correlations among the
columns of $\Psi$. The next experiment examines the robustness of UPDASC on  the
 correlation parameter $\mu $  in data generation setting (i).

\begin{exam}\label{exam:condition}
The matrix $\Psi\in \mathbb{R}^{300\times1000}$ is generated according to setting (i) with $\mu = 0.1:0.2:0.9$,
and the true signal $x^{\dag}$ contains $10$ nonzero elements.
\end{exam}

Like before, we evaluate Algorithm \ref{alg:pdasc} (i.e., UPDASC) by CPU time (in seconds), relative error (in the
$\ell^2$-norm), which are  computed from ten independent realizations of the problem setup. Table
\ref{tab:condition} presents both  CPU time (in second) and reconstruction error are fairly robust with respect to the
correlation  parameter $\mu$ ranging   from $0.1$ to $0.9$.

\begin{table}[H]
  \caption{Results for Example \ref{exam:condition}: robustness on $\mu$.}
 \label{tab:condition}
  \vspace{-0.3cm}
  \begin{center}
  \scalebox{0.9}{
  \begin{tabular}{ccccccccccccccc}
  \hline\hline
 \multicolumn{2}{c}{\quad   $\ell^0$}
  && \multicolumn{2}{c}{\quad   bridge }
  &&
  \multicolumn{2}{c}{\quad  capped-$\ell^1$} &&\multicolumn{2}{c}{ \quad SCAD} && \multicolumn{2}{c}{ \quad MCP}\\
  \cline{2-3}  \cline{5-6}   \cline{8-9} \cline{11-12}\cline{14-15}
  $\mu$       &Time         & RE                &   &Time         & RE                & & Time       & RE           && Time       & RE          && Time       & RE            \\
   $0.1$       &         2.80e-2     &6.00e-3      &    &      2.40e-2     &6.00e-3      & &    3.08e-2       &5.70e-3   &&    4.52e-2       &5.70e-3   &&  3.62e-2     &5.70e-3      \\
   $0.3$       &         1.88e-2     &5.00e-3      &    &      1.69e-2     &5.00e-3      & &    1.55e-2       &5.70e-3   &&    2.41e-2       &5.70e-3   &&  1.88e-2     &5.70e-3      \\
   $0.5$       &         1.88e-2     &5.00e-3      &    &      1.41e-2     &5.00e-3      & &    1.56e-2       &5.70e-3   &&    2.29e-2       &5.70e-3   &&  1.90e-2     &5.70e-3       \\
   $0.7$       &         1.75e-2     &6.02e-3      &    &      1.56e-2     &6.02e-3      & &    1.51e-2       &5.90e-3   &&    2.13e-2       &5.90e-3   &&  1.96e-2     &5.90e-3            \\
   $0.9$       &         1.95e-2     &6.63e-3      &    &      1.60e-2     &6.75e-3      & &    1.55e-2       &6.63e-3   &&    2.15e-2       &6.63e-3   &&  1.99e-2     &6.63e-3            \\
  \hline
  \hline \end{tabular}} \end{center}
\end{table}

Finally, we illustrate Algorithm \ref{alg:pdasc} on a genome-wide association study (GWAS) dataset.

\begin{exam}\label{exam:GWAS}
This test applies UPDASC to high-density lipoprotein (HDL) in NFBC1966 study~\cite{sabatti2009genome}. The NFBC1966 data set
contains information for 5,402 individuals with a selected list of phenotypic data including HDL and 364,590 single
nucleotide polymorphisms (SNPs). We perform strict quality control on data using PLINK~\cite{purcell2007plink}.
In the experiments, we exclude individuals having discrepancies between the reported sex and the sex determined from the X
chromosome, and exclude SNPs with a minor allele frequency less than 1$\%$, having missing values in more than 1$\%$
of the individuals or with a Hardy-Weinberg equilibrium p-value below 0.0001. After conducting strict quality control,
5,123 individuals with 9,114 SNPs in NFBC1966 on chromosome 16 are retained for the further analysis.
\end{exam}

The numbers of SNPs identified across $\ell^{0}$, $\ell^{1/2}$, SCAD, MCP and capped-$\ell^{1}$ are listed in Table~\ref{realdata3},
where the diagonals on the table are numbers of variables selected via different penalties and off diagonals are the
numbers of variables in the intersections of support sets determined across two different penalties. The solution pathes of different penalties
are presented in Fig. \ref{fig:GWAS}. Among all identified SNPs, rs3764261 and rs7499892 near gene {\it CETP} are
found to be associated with HDL in prior studies~\cite{willer2008newly,zemunik2009genome,wang2013cetp}.

\begin{table}[!ht]
  \caption{Results for Example \ref{exam:GWAS}. }
\label{realdata3}
\vspace{-0.3cm}
\begin{center}
	\begin{tabular}{cccccccc}
		\hline \hline
		         & $\ell^{0}$ & $\ell^{1/2}$ &SCAD & MCP & capped-$\ell^{1}$\\
		\hline
	   $\ell^{0}$  & 18 & 10 & 18 & 18 & 18\\
	   $\ell^{1/2}$ &   & 11 & 10 & 10 & 10\\
	   SCAD & & & 27 & 27 & 27 &\\
	   MCP & &  &  & 27 & 27\\
	   capped-$\ell^{1}$ & & & & &  27 \\
		\hline
	\end{tabular}
\end{center}
\end{table}

\begin{figure}[hbt!]
	\centering
   \begin{tabular}{cc}
   \includegraphics[width=0.43\textwidth]{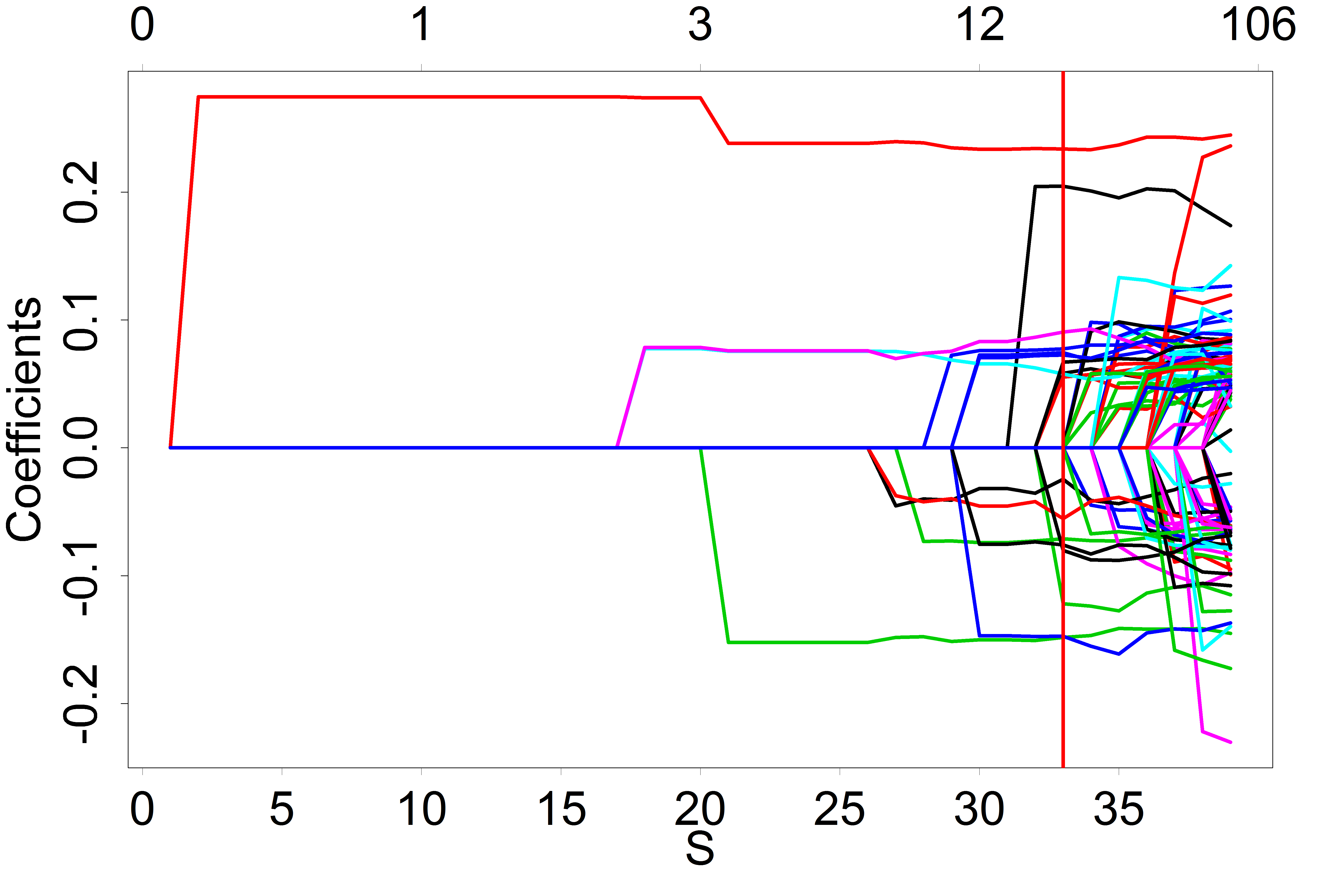} & \includegraphics[width=0.43\textwidth]{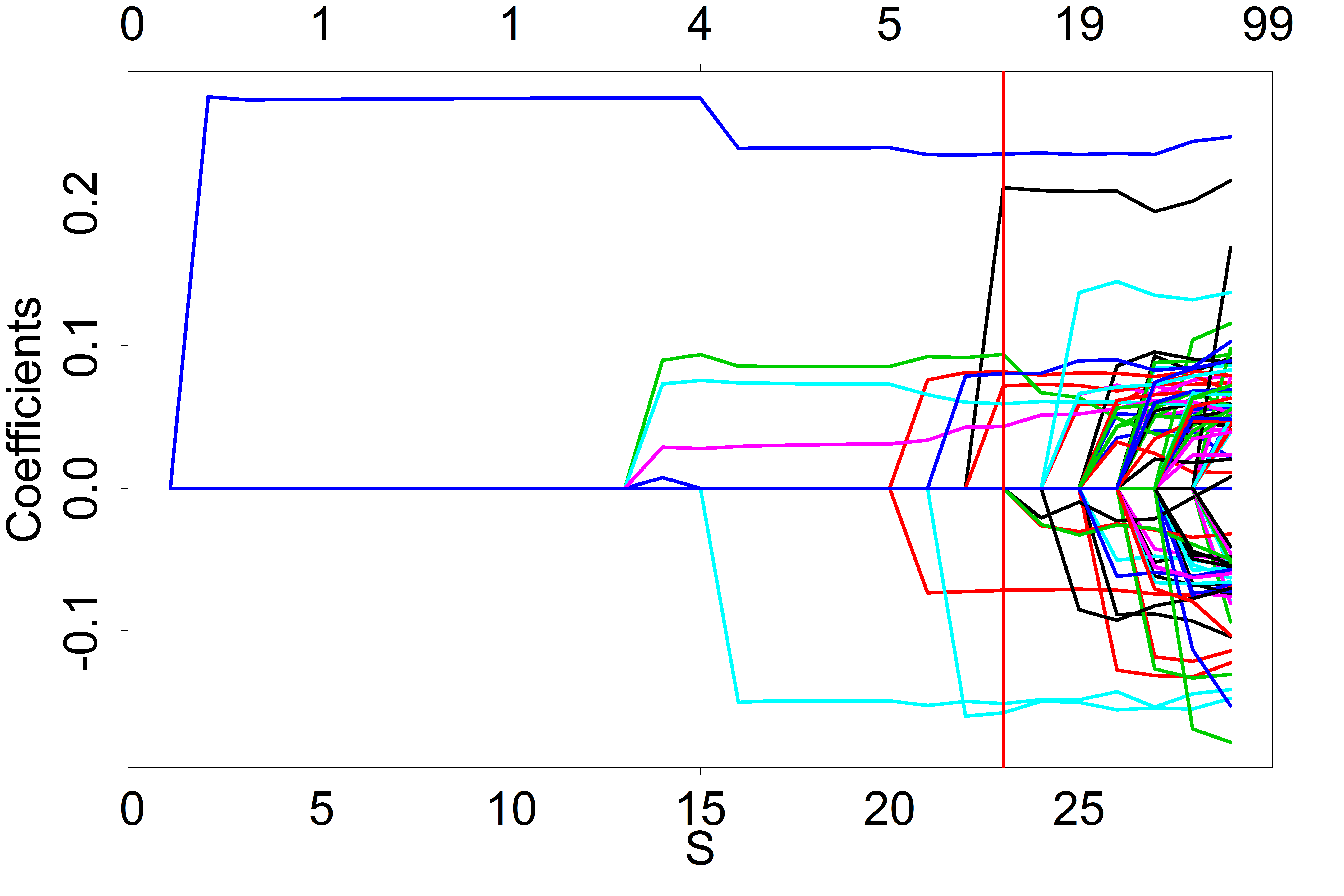}\\
    (a) $\ell^{0}$ & (b) $\ell^{1/2}$\\
   \includegraphics[width=0.43\textwidth]{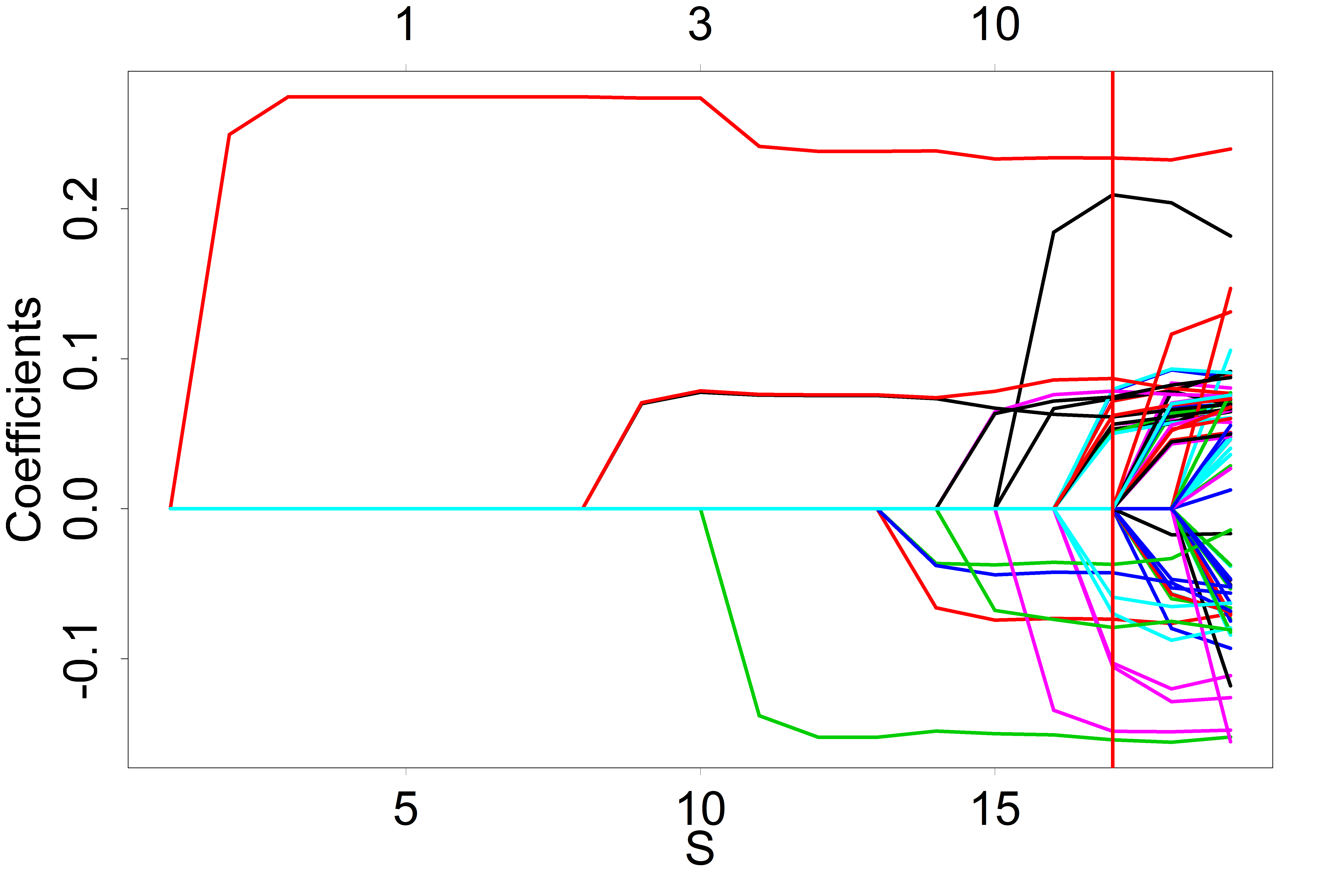} & \includegraphics[width=0.43\textwidth]{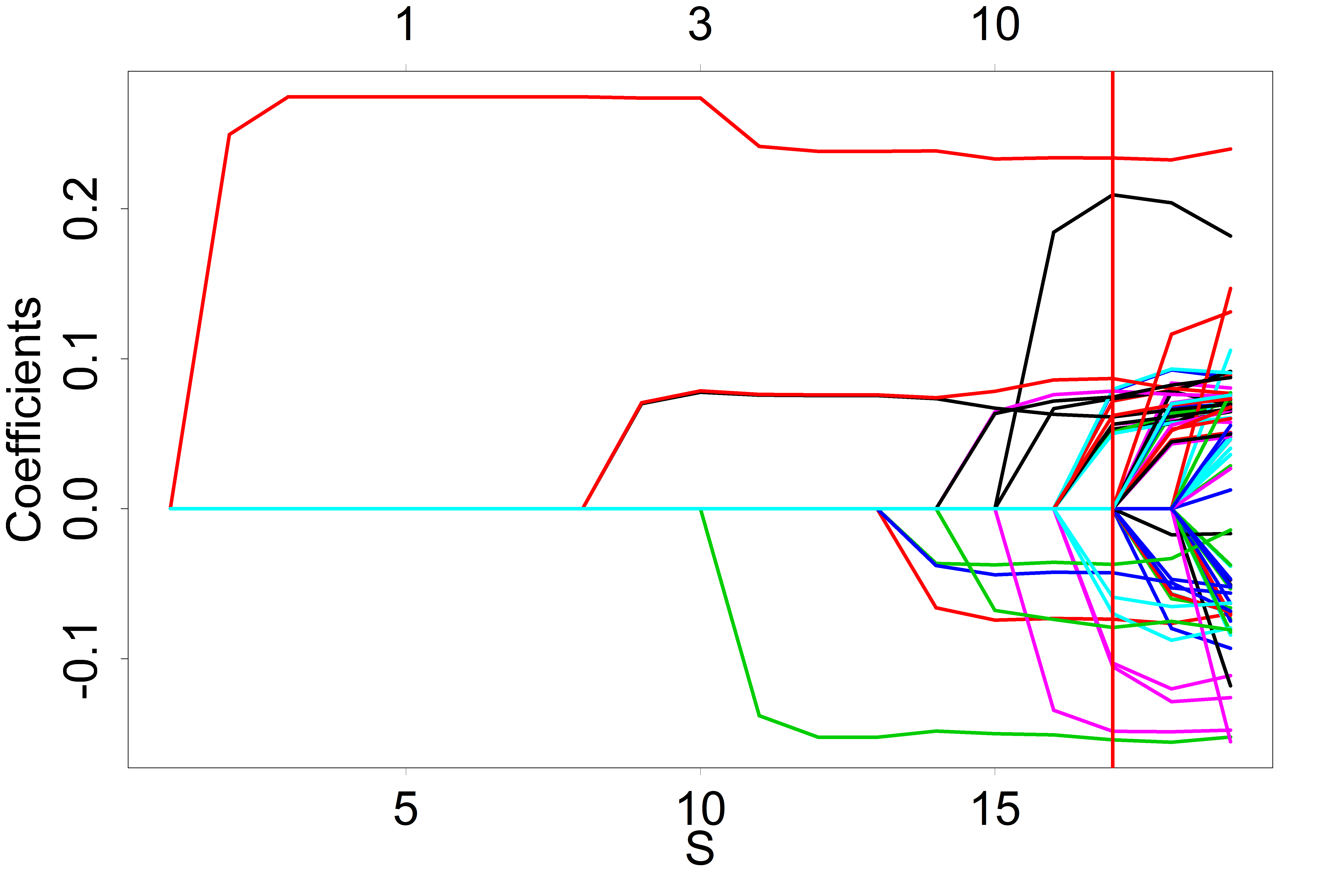}\\
   (c) SCAD & (d) MCP \\
   \includegraphics[width=0.43\textwidth]{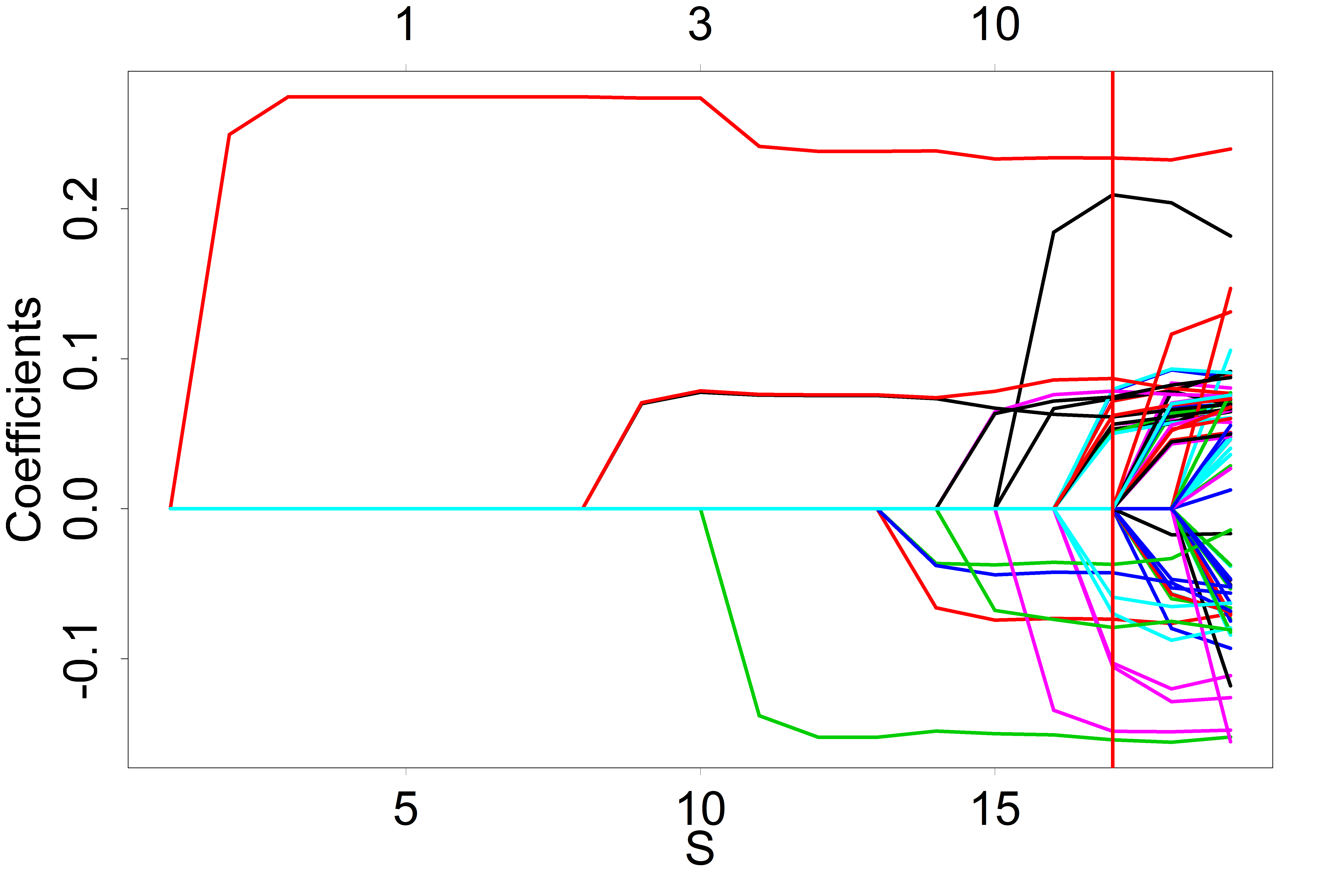}\\
   (e) capped-$\ell^{1}$
   \end{tabular}
	\caption{Solution path for example \ref{exam:GWAS} }\label{fig:GWAS}
\end{figure}

\section{Conclusions}

In this work, we have developed a unified PDAS algorithm for a class of popular nonconvex penalized regression
problems arising in high-dimensional statistics and sparse signal recovery, including the
$\ell^0$, bridge, capped-$\ell^1$, smoothly clipped absolute deviation and minimax concave
penalty. Theoretically, we established the existence of a global minimizer, and derived a necessary
optimality condition for a global minimizer, based on the associated thresholding operator.
The solutions to the necessary optimality condition are always coordinate-wise minimizers,
and further, we provided verifiable sufficient conditions for a coordinate-wise minimizer
to be a local minimizer. Meanwhile, the necessary optimality condition and the active set
can be reformulated using both primal and dual variables, which lends itself to a primal-dual
active set algorithm. One distinct feature of the algorithm is that at each iteration, it involves only solving a least-squares problem
on the active set, which is usually of much smaller size, and merits a local superlinear convergence, and thus when coupled with a
continuation strategy, the procedure is very efficient and accurate.  The global convergence
of the overall UPDASC was shown under suitable restricted isometry property on the design matrix.
The efficiency and accuracy of UPDASC combined with a new tuning parameter selection rule are
clearly demonstrated by extensive numerical experiments, including real data.

There are several avenues for further study. First, for very ill-conditioned design matrices, which
are characteristic of high-dimensional problems with highly correlated covariates, the linear
systems involved in the PDAS algorithm can be challenging to solve directly, and extra regularization
might be necessary. The extra regularization can be achieved by either penalization or early stopping.
This motivates further researches on related theoretical issues, especially stability and error estimates. Second, in some
practical applications, the design matrix $\Psi$ is only implicitly given where only matrix-vector multiplication
is available. This necessitates developing iterative linear solvers, and the study of inexact inner
iterations, e.g., especially convergence properties. Last, the extensions of the UPDASC algorithm to structured sparsity, e.g., group sparsity
penalty and the matrix analogues, are also of immense current interest.

\section*{Acknowledgements}

The authors are very grateful to the anonymous referee, the associate editor and the editor for their
helpful comments, which have led to a significant improvement on the quality of the paper.
The research of Y. Jiao is supported by National Science Foundation of  China Grant No. 11501579 and
National Science Foundation of  Hubei Province Grant No. 2016CFB486, and X. Lu is supported by National
Science Foundation of  China Grants Nos. 11471253 and 91630313. The research of J. Liu is supported by
Duke-NUS Graduate Medical School WBS: R-913-200-098-263 and MOE2016- T2-2-029 from Ministry of Eduction,
Singapore. The research of C. Yang is supported in part by grant No. 61501389 from National Science
Funding of China, grants No. 22302815, No. 12316116 and No. 12301417 from the Hong Kong Research Grant
Council, startup grant R9405 from The Hong Kong University of Science and Technology.

\bibliographystyle{abbrv}
\bibliography{pdas}

\newpage
\section*{Supplementary Materials}
\textbf{Appendix:} The supplementary materials contain the proofs of the theoretical results and one counterexample. 

\appendix
\section{Appendix} In the appendix, we prove Propositions  \ref{thm:Exist} and \ref{thm:thresholding} and Theorems  \ref{thm:localmin} and \ref{thm:convergence}.
and give one counterexample to shed light on the conditions  in Theorem  \ref{thm:localmin}.
\subsection{Proof of Proposition \ref{thm:Exist}}\label{app:Exist}
First, we show a technical lemma. 
\begin{lemma}\label{lem:1}
Let the function $\rho(t):\mathbb{R}\to\mathbb{R}^+\cup\{0\}$ satisfy:
\begin{itemize}
\item[$(\mathrm{i})$] $\rho$ is even with $\rho(0) = 0$, nondecreasing for
   $t\geq 0$, and lower semi-continuous.
\item[$(\mathrm{ii})$] $\rho(t)$ is a constant when $|t| \geq t_0$ for some $t_0>0$.
\end{itemize}
Then for any given subspace ${\cal N}\subset \mathbb{R}^p$, the following statements hold.
\begin{itemize}
\item[$(\mathrm{a})$] For any $x\in \mathcal{N}^{\perp}$, there exists an element $z\in\mathcal{N}$ such that
$\sum_{i=1}^p\rho(x_i+z_i) = \inf_{z\in{\cal N}} \sum_{i=1}^p\rho(x_i+z_i)$.
\item[$(\mathrm{b})$] Let $SE(x) = \mathop\mathrm{arg}\min_{z\in\cal N}\sum_{i=1}^p\rho(x_i+z_i)$. Then the function $h:{\cal N}^\perp\mapsto \mathbb{R}$, $h(x) = \inf_{z\in SE(x)} \|z\|$ maps a bounded set to a bounded set.
\end{itemize}
\end{lemma}
\begin{proof}
To show part (a), let $m$ be the dimension of $\cal N$, and $S\in \mathbb{R}^{p\times m}$
be a column orthonormal matrix whose columns form a basis of $\mathcal{N}$. We denote the
rows of $S$ by $\{\tilde{s}_i\}_{i=1}^p$.  
Then any $z\in\cal{N}$ can be written as $z = Sw$ for
some $w\in\mathbb{R}^m$. Let $\{w_k\}\subset\mathbb{R}^m$ be a minimizing sequence to $\inf_{z\in \cal N}
\sum_{i=1}^p\rho(x_i+z_i)$ under the representation $z=Sw$. We claim that there exists a $w \in\mathbb{R}^m$ such that
\begin{equation*}
  \sum_{i=1}^p\rho(x_i+ (S w)_i) \leq \lim_{k\rightarrow\infty}\sum_{i=1}^p\rho(x_i+(Sw_k)_i).
\end{equation*}

First, if there is a bounded subsequence of $\{w_k\}$, the existence
of a minimizer follows from the lower semicontinuity of $\rho$. Hence we assume that
$\|w_k\|\rightarrow \infty$ as $k\to\infty$. Then we check the scalar sequences $\{w_k\cdot
\tilde{s}_i\}_k$, $i\in\mathbb{I}\equiv\{1,...,p\}$. For any $i \in \mathbb{I}$,
if there is a bounded subsequence of $\{w_k\cdot \tilde{s}_i\}$, we may pass to a
convergent subsequence and relabel it to be the whole sequence. In this way, we divide
the index set $\mathbb{I}$ into two disjoint subsets $\mathbb{I}'$ and $\mathbb{I}''$ such that for every
$i\in \mathbb{I}'$, $w_k\cdot\tilde{s}_i$ converges, whereas for every $i\in \mathbb{I}''$,
$|w_k\cdot\tilde{s}_i|\rightarrow\infty$. Let $\mathcal{L} = \mathrm{span}\{\tilde{s}_i\}_{i\in
\mathbb{I}'}$, and decompose $w_k$ into $w_k = w_{k,\cal L} + w_{k,\cal L^\perp}$, with $w_{k,\cal L}
\in\mathcal{L}$ and $w_{k,\cal L^\perp}\in\mathcal{L}^\perp$.

If the index set $\mathbb{I}'$ is empty, then by the monotonicity of the function $\rho(t)$,
one can verify directly that the zero vector $0$ is a minimizer. Otherwise, by the definition of $\mathbb{I}'$ and the monotonicity of $\rho(t)$, for any
$i\in \mathbb{I}'$, $w_k\cdot\tilde{s}_i = w_{k,\cal L}\cdot\tilde{s}_i$, and for any $i\in
\mathbb{I}''$, $\limsup_{k\rightarrow\infty}\rho(x_i + w_{k,\mathcal{L}}\cdot\tilde{s}_i) \leq
\lim_{k\rightarrow\infty}\rho(x_i + w_k\cdot\tilde{s}_i)$. Hence $\{w_{k,\cal L}\}$ is also
a minimizing sequence. Next we prove that $\{w_{k,\cal L}\}$ is bounded.
To this end, we let $M$ be the submatrix of $S$, consisting of rows whose indices are
listed in $\mathbb{I}'$.
It follows from the definition of $w_{k,\cal L}$ and the convergence of $w_{k,\cal L}\cdot\tilde{s}_i$
for $i\in\mathbb{I}'$ that
$w_{k,\cal L} \in {\cal L} = \mbox{Range}(M^t)$ and $M w_{k,\cal L}$ is bounded.
Now let $M = U^t\Sigma V$ be the singular value decomposition of $M$. Then
for any $w\in \mbox{Range}(M^t)$, i.e., $w=M^tq$, there holds
\begin{equation*}
  \|Mw\|^2 = \|\Sigma\Sigma^t Uq\|^2 = \sum_{\sigma_i >0} \sigma_i^4 (Uq)_i^2 \geq
  \sigma_M^2\sum_{\sigma_i >0} \sigma_i^2 (Uq)_i^2 =  \sigma_M^2\|\Sigma^t Uq\|^2 = \sigma_M^2\|w\|^2,
\end{equation*}
where $\sigma_M$ is the smallest nonzero singular value of $M$. Hence the sequence
$\{w_{k,\cal L}\}$ is bounded, from which the existence
of a minimizer follows. This shows part (a).

By the construction in part (a), the set $SE(x)$ is nonempty, and by the lower semi-continuity of $\rho$, it
is closed. Hence, there exists an element $z(x)\in SE(x)$ such that $\|z(x)\| = \inf_{z\in
SE(x)} \|z\|$. We claim that the map $x\mapsto z(x)$ is bounded.
To this end, let $D=\|x\|_\infty$, and recall the representation $z = Sw$ and $(x+z)_i = x_i + w\cdot \tilde{s}_i$.
Denote by $\mathbb{I}'
= \{i\in \mathbb{I}: |w\cdot\tilde{s}_i| \leq D + t_0\}$, and let
${\cal L} = \mathrm{span}\{\tilde{s}_i\}_{i\in \mathbb{I}'}$, $w= w_{\cal L} + w_{\cal
L^\perp}$, with $w_\mathcal{L} \in {\cal L}$ and $w_{\cal L^\perp} \in{\cal L}^\perp$. Then
the argument in part (a) yields
\begin{equation*}
\left.\begin{array}{ll}\tilde{s}_i\cdot w_\mathcal{L} = \tilde{s}_i\cdot w & i\in\mathbb{I}'\\
\rho(x_i+ \tilde{s}_i\cdot w_\mathcal{L}) \leq \rho(t_0) = \rho(x_i + \tilde{s}_i\cdot w) & i\in\mathbb{I}\backslash\mathbb{I}' \end{array}\right\}\Rightarrow \tilde{z}(x) = Sw_{\mathcal{L}}(x)\in SE(x),
\end{equation*}
and
\begin{equation*}
  \|Sw_\mathcal{L}(x)\| \leq C_{\mathbb{I}'} (D+t_0),
\end{equation*}
where the constant $C_{\mathbb{I}'}$ depends only on the smallest nonzero singular
value of the submatrix whose rows are given by $\tilde{s}_i$, $i\in\mathbb{I}'$. Therefore,
\begin{equation*}
  \begin{aligned}
  \sup\limits_{\|x\|_{\infty} \leq D} \inf\limits_{z\in SE(x)} \|z\| & = \sup\limits_{\|x\|_{\infty} \leq D} \|{z}(x)\| \leq \sup\limits_{\|x\|_{\infty} \leq D} \|\tilde{z}(x)\| \\
   &= \sup\limits_{\|x\|_{\infty} \leq D}\|Sw_\mathcal{L}(x)\| \leq \sup\limits_{\mathbb{I}'} C_{\mathbb{I}'} (D+t_0).
  \end{aligned}
\end{equation*}
The factor $\sup_{\mathbb{I}'} C_{\mathbb{I}'}$ is over
finitely many numbers, which concludes the proof.
\end{proof}

\noindent
\textbf{Proof of Proposition  \ref{thm:Exist}}
\begin{proof}
We discuss the cases separately.

\noindent (i) bridge. The proof is straightforward due to the coercivity of the penalty.

%
\noindent(ii) $\ell^0$, capped-$\ell^1$, SCAD and MCP.  First, all these
penalties satisfy the assumptions in Lemma \ref{lem:1}. Let $\mathcal{N} = \mathrm{Ker}(\Psi)$,
then $\Psi$ is coercive over $\mathcal{N}^\perp$, and
$\mathrm{Ker}(\Psi)^\perp = \mathrm{Range}(\Psi^t)$.
Since the functional $J$ is bounded from below by zero, the infimum $\textrm{INF} = \inf J(x)$ exists
and it is finite; further, by the very definition of the infimum INF,
there exists a minimizing sequence, denoted by $\{x^k\}\subset\mathbb{R}^p$, to
\eqref{eqn:model}, i.e., $\lim_{k\rightarrow\infty}J(x^k) = \textrm{INF}$
\cite[Section 39, pp. 193]{GelfandFomin:1963}. We decomposed $x^k$ into $x^k =
P_\mathcal{N} x^k + P_{\mathcal{N}^\perp} x^k =: u^k +
v^k$, where $P_{\cal N}$ and $P_{{\cal N}^\perp}$ denote the orthogonal projection into
$\cal N$ and ${\cal N}^\perp$, respectively. By the construction
of the set $SE(v^k)$ in the proof of Lemma \ref{lem:1}, with the minimum-norm element $\tilde{u}^k\in SE(v^k)$ in
place of $u^k$, the sequence $\{v^k+\tilde{u}^k\}$ is still minimizing.
By the coercivity, $\{v^k\}$ is bounded, and
hence $\{\tilde{u}^k\}$ is also bounded by Lemma \ref{lem:1}(b). Upon passage to a convergent
subsequence, the lower semi-continuity of $J$ implies the existence of a minimizer.
\end{proof}

\subsection{Proof of Proposition \ref{thm:thresholding}}\label{app:thresholding}
We first prove Lemma \ref{thm:tT} and Lemma \ref{lem:proximal}.

\medskip\noindent
\textbf{Proof of Lemma \ref{thm:tT}}\label{app:tT}
\begin{proof}
We compute $(t^*,T^*)$ for the five penalties separately.

\noindent (i) $\ell^0$. $g(t) = \frac{t}{2} + \frac{\lambda}{t}$ for $t>0$ and $g(0) = +\infty$. Hence
   $t^* = \sqrt{2\lambda}$ and $T^* = g(t^*) = \sqrt{2\lambda}$.

\noindent (ii) bridge. $g(t) = \frac{t}{2} + \lambda t^{\tau-1}$ for $t>0$ and
   $g(0) = +\infty$. Direct computation gives $t^* = \left(2\lambda(1-\tau)
   \right)^{\frac{1}{2-\tau}}$, and $T^* = g(t^*) = (2-\tau)\left[2(1-\tau)
   \right]^{\frac{\tau-1}{2-\tau}} \lambda^{\frac{1}{2-\tau}}$.

\noindent (iii) capped-$\ell^1$. Then the function $g(t)$ is given by
     \begin{equation*}
        g(t) = \frac{t}{2} + \left\{\begin{array}{ll} \frac{\lambda^2 \tau}{t}, & t \geq \lambda\tau, \\[1ex]
        \lambda, &  0 \leq t \leq \lambda\tau. \end{array}\right.
     \end{equation*}
In the interval $[0,\lambda\tau]$, $0$ is the minimizer of $g(t)$ with a minimum value
$\lambda$, whereas in the interval $[\lambda\tau,\infty)$, the minimum value is $\lambda\sqrt{2\tau}$, which is greater
than $\lambda$. Hence $t^* = 0$, and $T^* = \lambda$.

\noindent (iv) SCAD. Then the function $g(t)$ is given by
\begin{equation*}
   g(t) = \frac{t}{2} + \left\{\begin{array}{ll} \frac{\lambda^2 (\tau +1 )}{2t}, & t \geq \lambda\tau, \\[1.2ex]
   \frac{\lambda\tau t - \frac{1}{2}(t^2 + \lambda^2)}{(\tau-1)t},  & \lambda \leq t\leq \lambda\tau, \\[1.2ex]
   \lambda, & 0\leq t \leq \lambda. \end{array}\right.
\end{equation*}
It can be verified directly that the minimizer of $g(t)$ in the intervals $[0,\lambda]$, $[\lambda,\lambda\tau]$,
$[\lambda\tau,\infty)$ is given by $0$, $\lambda$, $\lambda\sqrt{\tau+1}$, respectively. Hence $t^*
= 0$, and $T^* = \lambda$.

\noindent(v) MCP. Then the function $g(t)$ is given by
\begin{equation*}
   g(t) =  \frac{t}{2} + \left\{\begin{array}{ll}\frac{\lambda^2\tau}{2t}, & t \geq \lambda\tau, \\[1.2ex]
   \lambda -\frac{t}{2\tau}, & 0 \leq t \leq \lambda\tau. \end{array}\right.
\end{equation*}
Analogous to the case of the SCAD, we can obtain $t^* = 0$, and $T^* = \lambda$.
\end{proof}

\medskip\noindent
\textbf{Proof of Lemma \ref{lem:proximal}}
\begin{proof}
By the lower-semicontinuity and coercivity of the function ${(u-v)^2}/{2} + \rho(u)$,
it has at least one minimizer. Next one observes that
\begin{align*}
u^\ast  \in \mathop\textrm{argmin}_{u\in \mathbb{R}} \left({(u-v)^2}/{2} + \rho(u)\right)\quad
\Leftrightarrow\quad  u^\ast  \in \mathop\textrm{argmin}_{u\in \mathbb{R}} \left({u^2}/{2} - uv + \rho(u) \right).
\end{align*}
First, if $u^\ast=0$, then for any $u\neq 0$, ${u^2}/{2} - uv + \rho(u) = u (g(u) - v)$,
which implies that $u$ and $g(u)-v$ have the same sign. That is,
\begin{equation*}
 u > 0 \Rightarrow g(u) - v \geq 0, \quad \forall u>0, \quad \mbox{ then }
v \leq \inf\limits_{u > 0} g(u)= T^{*},
\end{equation*}
and
\begin{equation*}
 u < 0 \Rightarrow  g(u) - v \leq 0, \quad\forall u<0, \quad\mbox{ then }
 -v \leq \inf\limits_{ u < 0} -g(u) = \inf\limits_{ u < 0} g(-u)= T^{*}.
\end{equation*}
From these observations it follows that $|v|\leq T^{*}$. This shows assertion (a).
Second, let $G(u)=u (g(u) - v)$ for $u\neq 0$ and $G(0)=0$. For $|v|<T^{*}$, since
$$u>0\Rightarrow g(u)
\geq T^* > v \quad \mbox{and}\quad u< 0 \Rightarrow g(u) = - g(-u) \leq -T^* < v,$$
then $G(u)>0$ when $u\neq 0$,
which implies $0$ is the only minimizer. This shows (b).
Last, for $|v|=T^{*}$, by arguing analogously to (b) for $u>0$ and $u<0$, we have
$G(u)\geq 0$. Then $u^\ast$ satisfies that $G(u^\ast)=0$, i.e., $u^\ast=0$  or $g(u^\ast)=\sgn(v) T^{*}$.
\end{proof}

Now we can state the proof of Proposition \ref{thm:thresholding}.
\begin{proof}
We discuss only the case $v>0$, for which
$u^* \geq 0$. The case $v<0$ can be treated similarly.

\noindent (i) $\ell^0$. By Lemma \ref{lem:proximal}, if $|v| > T^*$, then $u^*\neq 0$ which implies
the minimizer $u^*$ is $ v$, from which the formula of $S_{\lambda}^{\ell^0}$ follows (see also \cite{ItoKunisch:2011}).

\noindent(ii) bridge. Let $G(u) = \frac{u^2}{2} + \lambda u^\tau - uv $ for $u\geq 0$.
Its first- and second derivatives are given by
\begin{equation*}
G'(u) = u + \lambda\tau u^{\tau -1} - v\quad\mbox{and} \quad G''(u)= 1 + \lambda \tau (\tau -1) u^{\tau -2}.
\end{equation*}
Clearly, $G'(u)$ is convex with $G'(0+) = G'(+\infty) = +\infty$.  Hence, $G'(u)$ has
at most two real roots, and $G(u)$ is either monotonically increasing or has three
monotone intervals.
This and Lemma \ref{lem:proximal} yield the expression $S_{\lambda,\tau
}^{\ell^\tau}$. Generally there is no closed-form expression for $S^{\ell^\tau}_{
\lambda,\tau}(v)$. For $|v|>T^*$, the unique minimizer to $G(u)$
is the larger root of $G'(u)$ (the other root is a local maximizer) (see also \cite{ItoKunisch:2011,GongZhangLuHuangYe:2013}).

\noindent(iii) capped-$\ell^1$. Let
\begin{equation*}
    G(u) = \left\{\begin{array}{ll} \frac{u^2}{2} - uv + \lambda^2 \tau, &u\geq \lambda\tau, \\[1.2ex]
     \frac{u^2}{2} - uv + \lambda u, & 0\leq u \leq \lambda\tau. \end{array}\right.
\end{equation*}
By Lemma \ref{lem:proximal}, for $|v|\leq \lambda$, we have $u^* = 0$.
We then assume $v > \lambda$. Simple computation shows
\begin{equation*}
  \begin{aligned}
    S^*_1 &:= \min_{u\geq\lambda\tau} G(u) = \lambda^2\tau - {v^2}/{2} \mbox{ at } u = v,\\
    S_2^* &:= \min_{u\in[0,\lambda\tau]}G(u) = -{(v - \lambda)^2}/{2} \mbox{ at } u = v-\lambda.
  \end{aligned}
\end{equation*}
Then we have
\begin{equation*}
   \left\{\begin{array}{ll}
     v > \lambda(\tau + \frac{1}{2}) \Rightarrow S_1^* < S_2^*, & u^* = v > \lambda\tau,\\[1.2ex]
     v < \lambda(\tau + \frac{1}{2}) \Rightarrow S_1^* > S_2^*, & u^* = v-\lambda < \lambda\tau,\\[1.2ex]
     v = \lambda(\tau + \frac{1}{2}) \Rightarrow S_1^* = S_2^*, & u^* = \lambda\tau \pm \frac{\lambda}{2},
   \end{array}\right.
\end{equation*}
whence follows the thresholding operator $S^{c\ell^1}_{\lambda,\tau}$ (see also \cite{GongZhangLuHuangYe:2013}).

\noindent(iv) SCAD. We define
\begin{equation*}
    G(u)  = \left\{\begin{array}{ll}
       G_1(u) \triangleq \frac{u^2}{2} - uv + \frac{\lambda^2 (\tau +1 )}{2}, &u \geq \lambda\tau, \\[1.3ex]
       G_2(u) \triangleq \frac{u^2}{2} - uv + \frac{\lambda\tau u - \frac{1}{2}(u^2 + \lambda^2)}{\tau-1},  & \lambda \leq u \leq \lambda\tau,\\[1.3ex]
       G_3(u) \triangleq \frac{u^2}{2} - uv + \lambda u, & 0\leq u\leq \lambda.
    \end{array}\right.
\end{equation*}
By Lemma \ref{lem:proximal}, $|v|\leq \lambda \Rightarrow u^* = 0$. We then assume $v > \lambda$.
The three quadratic functions $G_i(u)$ achieve their minimum at $u=v$,
$ u = \frac{(\tau-1)v - \lambda\tau}{\tau -2}$ and $u=v-\lambda$, respectively. Next we discuss
the three cases separately. First, if $v\geq \lambda\tau$, then $\frac{(\tau-1)v - \lambda\tau}{\tau -2}
\geq \lambda\tau$, which implies that $G_2(u)$ is decreasing on the interval $[\lambda,
\lambda\tau]$, it reaches its minimum at $\lambda\tau$. Similarly, $v-\lambda \geq \lambda$
implies that $G_3(u)$ reaches its minimum over the interval $[0,\lambda]$ at $\lambda$. Hence
\begin{equation*}
   \begin{aligned}
     \min\limits_{0\leq u\leq \lambda}G_3(u) & = G_3(\lambda) = G_2(\lambda)
      \geq \min\limits_{\lambda \leq u\leq\lambda\tau}G_2(u) = G_2(\lambda\tau) = G_1(\lambda\tau)
      \geq\min\limits_{u\geq \lambda\tau}G_1(u).
   \end{aligned}
\end{equation*}
Second, if $\lambda\tau \geq v\geq 2\lambda$, then $G_1$ is increasing on
$[\lambda\tau,\infty)$ and $G_3$ is decreasing on $[0,\lambda]$, and $\frac{(\tau-1)v
- \lambda\tau}{\tau -2} \geq \lambda\tau \in [\lambda,\lambda\tau]$. Hence
\begin{equation*}
   \begin{aligned}
      \min_{0\leq u\leq \lambda}G_3(u)=G_3(\lambda) = G_2(\lambda)\geq \min_{\lambda \leq u\leq\lambda\tau}G_2(u),\\
      \min_{\lambda \leq u\leq\lambda\tau}G_2(u) \leq G_2(\lambda\tau) = G_1(\lambda\tau)  = \min_{u\geq \lambda\tau}G_1(u).
   \end{aligned}
\end{equation*}
Third, if $2\lambda \geq v \geq \lambda$, similar argument gives that
\begin{equation*}
   \begin{aligned}
     \min\limits_{0\leq u\leq \lambda}G_3(u) & \leq G_3(\lambda) = G_2(\lambda)
      =\min\limits_{\lambda \leq u\leq\lambda\tau}G_2(u) \leq G_2(\lambda\tau)
      = G_1(\lambda\tau) = \min\limits_{u\geq \lambda\tau}G_1(u).
   \end{aligned}
\end{equation*}
This yields the thresholding operator $S_{\lambda,\tau}^\mathrm{scad}$ (see also
\cite{BrehenyHuang:2011,MazumderFriedmanHastie:2011,GongZhangLuHuangYe:2013}).

\noindent(v) MCP. Like before, we let
\begin{equation*}
   G(u)  = \left\{\begin{array}{ll}
     \frac{u^2}{2} - uv + \frac{1}{2}\lambda^2\tau, & u \geq \lambda\tau, \\[1.3ex]
     \frac{u^2}{2} - uv +\lambda u -\frac{u^2}{2\tau}, & 0\leq u \leq \lambda\tau.
   \end{array}\right.
\end{equation*}
Similar to case (iv), we obtain the expression for $S_{\lambda,
\tau}^\mathrm{mcp}$ (see also \cite{BrehenyHuang:2011,MazumderFriedmanHastie:2011,GongZhangLuHuangYe:2013}).
\end{proof}

\subsection{Proof of Theorem \ref{thm:localmin}}\label{app:localmin}

\begin{proof}
We prove Theorem \ref{thm:localmin} by establishing the inequality
\begin{equation}\label{eqn:ineqJ}
  J(x^*+\omega)\geq J(x^*)
\end{equation}
for small $\omega\in\mathbb{R}^p$, using the optimality condition and thresholding operator.

  \noindent (i) $\ell^0$. By Lemma \ref{lem:necopt} and using the thresholding operator $S_\lambda^{\ell^0}$,
we deduce that for $i\in{\cal A}$, $|x_i^*|\geq \sqrt{2\lambda}$. Further,
\begin{equation}\label{eqn:l0nec}
    0 = d^*_{\calA} = \Psi_{\calA}^t (y - \Psi_{\cal A} x^*_{\cal A})\quad
    \Leftrightarrow\quad x^*_{\calA} \in \mathop\textrm{argmin} \tfrac{1}{2}\|\Psi_{\calA} x_{\calA} - y\|^2.
\end{equation}
Now consider a small perturbation $\omega$, with $\|\omega\|_\infty<\sqrt{2\lambda}$, to
$x^*$. It suffices to show \eqref{eqn:ineqJ} for small $\omega$. Recall that $\omega_{\cal I}$ is
the subvector of $\omega$ whose entries are listed in the index set $\cal I$. If
$\omega_\mathcal{I}\neq 0$, then
\begin{equation*}
   \begin{aligned}
    J(x^* + \omega) - J(x^*) &\geq \tfrac{1}{2}\|\Psi x^* - y + \Psi \omega\|^2 -\tfrac{1}{2}\|\Psi x^* - y \|^2  + \lambda
      \geq \lambda - |(\omega, d^\ast )|,
   \end{aligned}
\end{equation*}
which is positive for small $\omega$. Meanwhile, if $\omega_{\calI} = 0$,
by \eqref{eqn:l0nec}, we deduce \eqref{eqn:ineqJ}.

\noindent(ii) bridge.
First note that on the active set $\calA$, $|x_i^*|\geq t^*=\left(2\lambda(1-\tau)
\right)^{\frac{1}{2-\tau}}$. Next we claim that if the minimizer $u^*$ of
$G(u) = \tfrac{u^2}{2} - uv + \lambda u^\tau$ is positive, then $G(u)$ is locally strictly
convex around $u^*$, i.e., for small $t$ and some $\theta>0$ such that
\begin{equation*}
   G(u^*+t) - G(u^*) =  G(u^* + t) - G(u^*) - G'(u^*)t  \geq \theta t^2.
\end{equation*}
To see this, we recall that $u^*$ is the larger root of $u+\lambda \tau u^{\tau-1}=v$ and
$v\geq T^*$. By the convexity of $u+\lambda \tau u^{\tau-1}$, $u^*(v)$ is
increasing in $v$ for $v \geq T^*$. Further, by the inequality
$u^*\geq t^*$, we have
\begin{equation*}
  \begin{aligned}
     G''(u^*) &= 1 - \lambda\tau(1-\tau) (u^*)^{\tau -2} \\
     & \geq 1 - \lambda\tau(1-\tau) (t^*)^{\tau -2} = 1 - \tfrac{\tau}{2}.
  \end{aligned}
\end{equation*}
In particular, the function $G(u)$ is locally strictly convex with $\theta=\frac{1}{2}
-\frac{\tau}{4} - \epsilon$, for any $\epsilon>0$. Hence for each $i\in\calA$ and small $t$, there holds
\begin{equation*}
   \begin{aligned}
    J(x^* + t e_i) - J(x^*) 
             & = \tfrac{1}{2}t^2 + (t\psi_i,\Psi x^* - y) + \lambda |x^*_i + t|^\tau - \lambda |x^*_i|^\tau \geq \theta t^2,
   \end{aligned}
\end{equation*}
i.e.,
\begin{equation*}
   -td_i^* + \lambda |x_i^* + t|^\tau - \lambda |x_i^*|^\tau \geq (\theta-\tfrac{1}{2}) t^2.
\end{equation*}
Consequently for small $\omega$, we have
\begin{equation*}
  \begin{aligned}
    J(x^* + \omega) - J(x^*)& = \tfrac{1}{2}\|\Psi \omega\|^2 - (\omega, d^*) + \sum_{i\in\calA}\lambda (|x_i^* + \omega_i|^\tau-|x_i^*|^\tau) + \lambda\sum_{i\in \calI} |\omega_i|^\tau \\
    &\geq  \tfrac{1}{2}\|\Psi \omega\|^2 -(\omega_{\cal I}, d^*_{\calI}) + \lambda\sum_{i\in \calI} |\omega_i|^\tau + (\theta - \tfrac{1}{2})\|\omega_{\calA}\|^2.
  \end{aligned}
\end{equation*}
Note the trivial estimates
\begin{equation*}
   -(\omega_{\calI}, d^*_{\calI}) \geq -\sum_{i\in\calI}|\omega_i|\|d^*_{\mathcal{I}}\|_\infty
\quad\mbox{and}\quad
     \tfrac{1}{2}\|\Psi \omega\|^2 
      \geq \tfrac{1}{2}\|\Psi_{\calA}\omega_{\cal A}\|^2 + (\omega_{\calA},\Psi^t_{\calA}\Psi_{\calI}\omega_{\calI}).
\end{equation*}
Further, by Young's inequality, for any $\delta>0$
\begin{equation*}
   (\omega_{\calA},\Psi^t_{\calA}\Psi_{\calI}\omega_{\calI})\geq -\delta \|\omega_{\calA}\|^2 - \tfrac{1}{4\delta}\|\Psi^t_{\calA}\Psi_{\calI}\omega_{\calI}\|^2\geq -\delta \|\omega_{\calA}\|^2 - C_\delta\|\omega_{\calI}\|^2.
\end{equation*}
Combing these four estimates together and noting $\theta=\frac{1}{2} - \frac{\tau}{4} - \epsilon$ yields
\begin{equation*}
  \begin{aligned}
    J(x^*+\omega)- J(x^*)&\geq  \left(\tfrac{1}{2}\|\Psi_{\calA}\omega_{\calA}\|^2 - (\tfrac{\tau}{4}
    + \epsilon + \delta)\|\omega_{\calA}\|^2\right) \\
      & \quad + \sum_{i\in \calI} |\omega_i|^\tau \left(\lambda-|\omega_i|^{1-\tau}\|d^*_{\mathcal{I}}\|_{\infty}-C_\delta|\omega_i|^{2-\tau}\right).
  \end{aligned}
\end{equation*}
The first term is nonnegative if $\epsilon$ and $\delta$ are small and
$\Psi$ satisfies \eqref{eqn:sec} with $\sigma(\mathcal{A})> \frac{\tau}{2}$.
The sum over $\calI$ is  nonnegative for small $\omega$, thereby showing \eqref{eqn:ineqJ}.\\

The proof of the rest cases is based on the identity
\begin{equation}\label{eqn:difJ}
    J(x^*+\omega)-J(x^*)= \tfrac{1}{2}\|\Psi \omega\|^2+\sum_i\underbrace{(\rho_{\lambda,\tau}(x_i^*+\omega_i)-\rho_{\lambda,\tau}(x_i^*)- \omega_id^*_i)}_{:=s_i}.
\end{equation}

\noindent(iii) capped-$\ell^1$.
We denote by ${\calA}_1 =\{i:|x_i^*| > \lambda\tau \}$, ${\calA}_2=\{i:\lambda\tau >
|x_i^*| > 0 \}$. By assumption $\{i: |x_i^*|=\lambda\tau\}=\emptyset$, hence
${\calI}=({\cal A}_1\cup{\cal A}_2)^c$. The optimality condition for $x^*$ and the
differentiability of $\rho_{\lambda,\tau}^{c\ell^1}(t)$ for $|t|\neq\lambda\tau$ yield
$d_i^*=0 $ for $i\in{\calA}_1$, and $d_i^*=\lambda\sgn(x_i^*)$
for $i\in {\calA}_2$. Thus, for $\omega$ small, there holds
\begin{equation*}
  s_i = \left\{\begin{array}{ll}
    0, & i\in{\calA}_1\cup\calA_2,\\ 
    \lambda |\omega_i|-\omega_id_i^*, & i\in\calI.
  \end{array}\right.
\end{equation*}
Now with the fact that for $i\in \calI$, $|d^*_i|\leq \lambda$, we deduce that
for small $\omega$, \eqref{eqn:ineqJ} holds. 

\noindent (iv) SCAD. 
Let $\calA_1=\{i: |x_i^*| > \lambda\tau\}$, $\calA_2= \{i: |x_i^*| \in [\lambda,\lambda\tau]\}$,
$\calA_3= \{i: |x_i^*|\in (0,\lambda)\}$, and $\calI =  (\cup\mathcal{A}_i)^c$.
Then 
the optimality of $x^*$ yields
\begin{equation*}
    d_i^*=\left\{\begin{array}{ll}
      0,& i\in \calA_1,\\ 
     \tfrac{\lambda\sgn(x_i^*)-x_i^*}{\tau-1}, & i\in \calA_2,\\ 
     \lambda\sgn(x_i^*), & i\in\calA_3,
  \end{array}\right.
\end{equation*}
and $|d_i^*|\leq\lambda$ on $\calI$. Then for small $\omega$
in the sense that for
$$i\in{\cal A}_{1}\Rightarrow |x_i^*+\omega_i| > \lambda\tau \quad \mbox{ and }\quad i\in{\cal A}_{3}\Rightarrow |x_i^*+\omega_i| \in (0, \lambda),$$
we obtain $ s_i = 0$, $ i\in\calA_1\cup\calA_3 $. For $i\in \calA_2$, we have two cases:
\begin{equation*}
    s_i\left\{\begin{array}{ll}
      =-\tfrac{w_i^2}{2(\tau-1)},&\quad\mbox{if }|x_i^* + \omega_i|\in[\lambda,\tau\lambda],\\
      \geq-\tfrac{w_i^2}{2(\tau-1)},& \quad \mbox{otherwise. } 
      \end{array}\right.
\end{equation*}
Finally for $i\in \calI$, $|d_i^*|<\lambda$ by assumption, and hence
\begin{equation*}
  s_i = \lambda|w_i|-d_i^*w_i \geq |w_i|(\lambda-|d_i^*|).
\end{equation*}
Combining these estimates with \eqref{eqn:difJ}, we arrive at
\begin{equation*}
  \begin{aligned}
    J(x^* + \omega) - J(x^*)
    &\geq  \tfrac{1}{2}\|\Psi\omega\|^2 - \tfrac{1}{2(\tau -1)}\|\omega_{\calA_2}\|^2+\sum_{i\in\calI}|\omega_i|(\lambda-|d_i^*|),
  \end{aligned}
\end{equation*}
Further, by Young's inequality, we bound
\begin{equation*}
  \begin{aligned}
    \tfrac{1}{2}\|\Psi\omega\|^2 
      \geq \tfrac{1}{2} \|\Psi_\calA\omega_\calA\|^2+ (\omega_\calA,\Psi_\calA^t\Psi_I\omega_I) 
      \geq (\tfrac{1}{2}-\epsilon)\|\Psi_\calA\omega_\calA\|^2 -C_\epsilon\|\omega_\calI\|^2.
  \end{aligned}
\end{equation*}
Consequently, there holds
\begin{equation*}
  J(x^*+\omega)-J(x^*)\geq (\tfrac{1}{2}-\epsilon)\|\Psi_\calA\omega_\calA\|^2-\tfrac{1}{2(\tau-1)}\|\omega_{\calA_2}\|^2 - \sum_{i\in\calI}
  \left(|\omega_i|(\lambda-|d_i^*|-C_\epsilon|\omega_i|\right).
\end{equation*}
If \eqref{eqn:sec} with $\sigma(\mathcal{A})>\frac{1}{\tau-1}$ and
$\|d^*_\mathcal{I}\|_{\infty} < \lambda$ hold, then \eqref{eqn:ineqJ} follows.

\noindent(v) MCP. The proof is similar to case (iv). We let $\calA_1=\{i:
|x_i^*|>\tau\lambda\}$, $\calA_2=\{i: 0<|x_i^*|\leq\tau\lambda\}$,
and $\calI = (\cup\calA_i)^c$. The differentiability of $\rho_{\lambda,
\tau}^\mathrm{mcp}(t)$ yields
\begin{equation*}
  d_i^* = \left\{\begin{array}{ll}
    0, & i\in\calA_1,\\ 
    \lambda\sgn(x_i^*)-x_i^*/\tau, & i\in\calA_2,
  \end{array}\right.
\end{equation*}
and on the set $\calI$, $|d_i^*|\leq\lambda$. Note that for small $\omega_i$,
there holds $s_i  =0$ for $i\in \calA_1$. Similarly, for
$i\in\calA_2$, there holds
\begin{equation*}
  s_i \geq  \left\{\begin{array}{ll} \tfrac{1}{2\tau}|\omega_i|^2,&i\in\calA_2,\\ 
    |\omega_i|(\lambda-|d_i^\ast|), & i\in\calI,
   \end{array}\right.
\end{equation*}
The rest of the proof is identical with case (iv), and hence omitted.
\end{proof}

\subsection{Explicit expression of $d^*_\calA$}\label{app:dual}
For a coordinate-wise minimizer $x^*$, we derive the explicit expression shown in Table
\ref{tab:activeset} for the dual variable
$d^* = \Psi^t( y - \Psi x^*)$ on the active set ${\cal A} = \{i: x_i^* \neq 0\}$.

\noindent(i) $\ell^0$. By the expression of $S_{\lambda}^{\ell^0}$, we have $d_i^* x_i^*
= 0$, and hence $d^*_i = 0$, for $i\in \calA$.

\noindent(ii) bridge. Since for $i\in\calA$, $J(x^*)$ is differentiable along the direction
$e_i$ at point $x_i^*$, the necessary optimality condition for $x_i^*$ reads $d_i^* -
\lambda\tau \frac{|x_i^*|^\tau}{x_i^*} = 0$.

\noindent(iii) capped-$\ell^1$. We divide the active set $\calA$ into $\calA=\cup_i
\calA_i$, with ${\cal A}_1= \{i: |x_i^* + d_i^*| > \lambda(\tau + \tfrac{1}{2})\}$,
 ${\cal A}_2 = \{i: \lambda < |x_i^* + d_i^*| < \lambda(\tau + \tfrac{1}{2})\}$,
and ${\cal A}_3 = \{i: |x_i^* + d_i^*| = \lambda(\tau + \tfrac{1}{2})\}$. Then
the definition of the operator $S^{c\ell^1}_{\lambda,\tau}$ gives the desired expression.

\noindent(iv) SCAD. We divide the active set $\calA$ into $\calA=\cup_i\calA_i$ with
 ${\cal A}_1 = \{i: |x_i^* + d_i^*|\geq \lambda\tau\}$, ${\cal A}_2 = \{i: \lambda\tau >
|x_i^* + d_i^*| > 2\lambda\}$, and ${\cal A}_3 = \{i:2\lambda \geq |x_i^* + d_i^*| > \lambda \}$.
Then it follows from the necessary optimality condition for $x_i^*$ that the desired
expression holds.

\noindent(v) MCP. Similar to case (iv), we divide  the active set $\calA$ into
$\calA = \cup_i\calA_i$ with ${\cal A}_1 = \{i: |x_i^* + d_i^*| \geq \lambda\tau\}$
and ${\cal A}_2 = \{i: \lambda <|x_i^* + d_i^*| < \lambda\tau\}$. Then the desired
expression follows from the optimality condition for $x_i^*$.

\subsection{A counterexample}\label{counterexample}

In this part, we construct a counterexample to show that a coordinate-wise minimizer is not
necessarily a local minimizer. The example is two-dimensional and with the MCP penalty.
The minimization problem reads:
\begin{equation*}
  \min_{x} J(x) =\tfrac{1}{2}\|y- \Psi x\|^2 +  \sum_{i=1}^2 \rho_{\lambda,\tau}(x_i)
\end{equation*}
with $\rho_{\lambda,\tau}$ being the MCP penalty. Consider the following design matrix $\Psi$, data $y$ and the vector $x^*$:
\begin{equation*}
  \Psi = \frac{1}{\sqrt{2}} \left(  \begin{array}{cc}
                 1 & 1 \\  1 & 1
      \end{array}\right),\quad
  y =\left(\begin{array}{c}
             1 \\  -1
  \end{array}  \right),\quad \mbox{and}\quad x^* = \tau \lambda \left(
   \begin{array}{c} 1 \\  -1  \end{array} \right) .
\end{equation*}
Simple computation shows that $\Psi x^*  = \textbf{0}$ and $\Psi^t y = \textbf{0}.$ By Lemma
\ref{lem:necopt}, $x^{*}$ is a coordinate-wise minimizer since it satisfies \eqref{eqn:nec}.
Next we show that, $x^{*}$ is not a local minimizer. To this end, let $x_t = x^{*} - t y$
with an arbitrarily small positive number $t$, then $\Psi x_t = \textbf{0}$. Straightforward
computation indicates
\begin{equation*}
  J(x^{*}) = \frac{1}{2}\|y\|^2 + \sum_{i=1}^2 \rho_{\lambda,\tau}(\lambda\tau)\quad \mbox{  and  }\quad
 J(x_t) = \frac{1}{2}\|y\|^2 + \sum_{i=1}^2 \rho_{\lambda,\tau}(\lambda\tau-t).
\end{equation*}
Then, $J(x_t)< J(x^{*})$ for any small positive number $t$.

\subsection{Proof of Theorem \ref{thm:convergence}}\label{app:convergence}
First we recall some estimates for the RIP constant $\delta_k$ (see, e.g.,
\cite{NeedellTropp:2009,TroppWright:2010}). Let $\cal A\cap \cal  B = \emptyset$ and $\delta_{|\cal A| +|\cal B|}$ exists, then
\begin{eqnarray*}
&\|\Psi_{\cal A}^t\Psi_{\cal A} x_{\cal A}\|\gtreqqless (1\mp \delta_{|{\cal A}|})\|x_{\cal A}\|, \quad &\|(\Psi_{\cal A}^t\Psi_{\cal A})^{-1} x_{\cal A}\|\gtreqqless \tfrac{1}{1\mp \delta_{|{\cal A}|}}\|x_{\cal A}\|, \\
&\|\Psi_{\cal A}^t\Psi_{\cal B}\|\leq \delta_{|{\cal A}|+|{\cal B}|}, \quad & \|\left[I - (\Psi_{\cal A}^t\Psi_{\cal A})^{-1}\right] x_{\cal A}\| \leq \tfrac{\delta_{|\cal A|}}{1- \delta_{|{\cal A}|}}\|x_{\cal A}\|.
\end{eqnarray*}
Given any index set ${\cal A} \subset {\cal A}^\dag$, we denote ${\cal I} = {\cal A}^c$ and $ {\cal B} = {\cal A}^\dag\backslash {\cal A}$, and further, let
\begin{equation*}
x_{{\cal A}} = (\Psi_{{\cal A}}^t\Psi_{{\cal A}})^{-1}(\Psi_{{\cal A}}^t y - p_{{\cal A}}), \quad d_{{\cal A}} = \Psi_{{\cal A}}^t( y - \Psi_{{\cal A}}x_{{\cal A}}),
\end{equation*}
Then we have $d_{{\cal A}} = p_{{\cal A}}$. By noting the trivial relation $y = \Psi_{\cal A}x_{\cal A}^\dag + \Psi_{\cal B}x_{\cal B}^\dag$, we deduce
\begin{equation}\label{equ:app0}
  \begin{aligned}
  \|x_{\cal A} - x_{\cal A}^\dag\| &\leq \|(\Psi^t_{\cal A}\Psi_{\cal A})^{-1}\Psi_{\cal A}^t\Psi_{\cal B}x_{\cal B}^\dag\| + \| (\Psi^t_{\cal A}\Psi_{\cal A})^{-1}p_{\cal A}\|\\
  & \leq \tfrac{\delta}{1-\delta}\|x_{\cal B}^\dag\| + \tfrac{1}{1-\delta}\|p_{\cal A}\|.
  \end{aligned}
\end{equation}
Then by appealing to the identity $d_i = \Psi_i^t(\Psi_{\cal B}x_{\cal B}^\dag - \Psi_{\cal A}(x_{\cal A} - x_{\cal A}^\dag))$ and \eqref{equ:app0}, we find
\begin{eqnarray}
&&\|x_{\cal A} + d_{\cal A}- x_{\cal A}^\dag\|  \leq \tfrac{\delta}{1-\delta}\|x_{\cal B}^\dag\| + \tfrac{\delta}{1-\delta}\|p_{\cal A}\| \triangleq h_{\cal A},\label{equ:app1}\\
&&|d_i| \geq |x_i^\dag| - \delta\|(x_{\cal A} - x_{\cal A}^\dag)\| -\delta\|x_{\cal B}^\dag\| \geq |x_i^\dag| - h_{\cal A}, \quad \forall i\in {\cal B},\label{equ:app2}\\
&&|d_i| \leq \delta\|(x_{\cal A} - x_{\cal A}^\dag)\| + \delta\|x_{\cal B}^\dag\|\leq h_{\cal A}, \quad \forall i\in {\cal I}^\dag. \label{equ:app3}
\end{eqnarray}
Next we define the index set $G_{\lambda,s}$ by
\begin{equation}\label{equ:sets}
  G_{\lambda,s} \triangleq \left\{
  \begin{array}{ll} \left\{i: |x_i^\dag| \geq \lambda s \right\} & \mbox{capped-}\ell^1,\ \mbox{SCAD}, \; \mbox{MCP}, \\
    \left\{i: |x_i^\dag| \geq (\lambda s)^{\tfrac{1}{2-\tau}}\right\} & \mbox{bridge}. \end{array} \right.
\end{equation}

The general strategy of the proof is similar to that in \cite{FanJiaoLu:2014,JiaoJinLu:2015}. It relies crucially on
the following monotonicity property on the active set. Namely, the evolution of the active set during the iteration can be precisely
controlled, by suitably choosing the decreasing factor $\rho$ and $s$.
\begin{lemma}\label{lem:monotone}
For $\rho\in(0,1)$ close to unity and some $s>0$, there holds
\begin{equation}\label{equ:key}
G_{\lambda,s} \subset {\cal{A}}_k \subset {\cal{A}}^\dag \Rightarrow G_{\rho\lambda,s} \subset {\cal{A}}_{k+1} \subset {\cal{A}}^\dag.
\end{equation}
\end{lemma}
\begin{proof}
Assume for some inner iteration $G_{\lambda,s} \subset {\cal{A}}_k \subset {\cal{A}}^\dag $. Let ${\cal A} = {\cal A}_k$ and ${\cal B} = {\cal A}^\dag \backslash {\cal A}$. First we derive upper bounds on the crucial term $h_{\cal A}$ in \eqref{equ:app1}. It follows from \eqref{eqn:bound} and the definition of $G_{\lambda,s}$ that
\begin{equation*}
h_{\cal A} \leq \left\{\begin{array}{ll}\frac{\delta}{1-\delta}\left(s\lambda\sqrt{|B|} + \lambda\sqrt{|A|}\right) & \mbox{capped-}\ell^1,\ \mbox {MCP},\\[1.2ex]
\tfrac{\delta}{1-\delta}\left(s\lambda\sqrt{|B|} + \lambda\tfrac{\tau}{\tau-1}\sqrt{|A|}\right)& \mbox{SCAD}, \\[1.2ex]
\tfrac{\delta}{1-\delta}\left((s\lambda)^{\frac{1}{2-\tau}}\sqrt{|B|} + (\lambda c_\tau)^{\frac{1}{2-\tau}}\sqrt{|A|}\right)& \mbox{bridge}, \end{array}\right.
\end{equation*}
where the constant $c_\tau = [2(1-\tau)]^{\tau -1}$. Upon noting $|A| + |B| = T$ and the elementary
inequality $a\sqrt{t} + b\sqrt{T-t} \leq \sqrt{a^2+b^2}\sqrt{T}$, we deduce
\begin{equation}\label{equ:ha}
h_{\cal A} \leq \left\{\begin{array}{ll}\frac{\delta}{1-\delta}\sqrt{s^2 +1}\sqrt{T}\lambda & \mbox{capped-}\ell^1, \;\mbox{MCP},\\[1.2ex]
\tfrac{\delta}{1-\delta}\sqrt{s^2 + \tfrac{\tau^2}{(\tau-1)^2} }\sqrt{T}\lambda& \mbox{SCAD}, \\[1.2ex]
\tfrac{\delta}{1-\delta}\sqrt{(\frac{s}{c_\tau})^{\frac{2}{2-\tau}}+1} \sqrt{T} (c_\tau\lambda) ^{\tfrac{1}{2-\tau}} & \mbox{bridge}. \end{array}\right.
\end{equation}
Now we prove \eqref{equ:key} for different penalties. In view of \eqref{equ:app1}-\eqref{equ:app3}, it suffices to show
$h_{\cal A} < T^*$ and $\rho s\lambda - h_{\cal A} > T^*$, where $T^*$ is given in Lemma \ref{thm:tT}.\\
{\bf Capped-$\ell^1$ and MCP:} Since $\delta < \frac{1}{\sqrt{5T}+1}$, $\frac{\delta}{1-\delta}\sqrt{5T} <1 $. Then we choose $s=2$ and $\rho\in (\frac{1+\frac{\delta}{1-\delta}\sqrt{5T}}{2},1)$. It follows from \eqref{equ:ha} that
\begin{eqnarray*}
&&h_{\cal A} \leq \tfrac{\delta}{1-\delta}\sqrt{5T} \lambda < \lambda \Rightarrow {\cal A}_{k+1} \subset {\cal A}^\dag,\\
&&\rho s\lambda - h_{\cal A} \geq 2\rho \lambda - \tfrac{\delta}{1-\delta}\sqrt{5T} \lambda > \lambda \Rightarrow G_{\rho\lambda,s}\subset {\cal A}_{k+1}.
\end{eqnarray*}
{\bf SCAD:} Like before, since $\delta < \frac{1}{\sqrt{8T}+1}$, we deduce $\frac{\delta}{1-\delta}\sqrt{8T} <1 $. We choose $s=2$ and $\rho\in (\frac{1+\frac{\delta}{1-\delta}\sqrt{8T}}{2},1)$. Then by \eqref{equ:ha} and noting $\tau>2 \Rightarrow \frac{\tau}{\tau-1} <2$, we obtain
\begin{eqnarray*}
&&h_{\cal A} \leq \tfrac{\delta}{1-\delta}\sqrt{4 + \tfrac{\tau^2}{(\tau-1)^2}}\sqrt{T} \lambda < \lambda \Rightarrow {\cal A}_{k+1} \subset {\cal A}^\dag,\\
&&\rho s\lambda - h_{\cal A} \geq 2\rho \lambda - \tfrac{\delta}{1-\delta}\sqrt{8T} \lambda > \lambda \Rightarrow G_{\rho\lambda,s}\subset {\cal A}_{k+1}.
\end{eqnarray*}
{\bf Bridge:} Recall $T^* = (2-\tau)(c_\tau\lambda)^{\frac{1}{2-\tau}}$, cf. Lemma \ref{thm:tT}. Since $\delta < \frac{2-\tau}{2-\tau + \sqrt{T[(4- 2\tau)^2 +1]}}$, let $\frac{s}{c_\tau} = (4-2\tau)^{2-\tau}$ and we have $\frac{\delta}{1-\delta}\sqrt{(\frac{s}{c_\tau})^{\frac{2}{2-\tau}}+1} \sqrt{T} \leq 2-\tau$. By choosing $\rho \in (\frac{2-\tau + \frac{\delta}{1-\delta}\sqrt{T[(4- 2\tau)^2 +1]}}{4-2\tau},1)$, we deduce
\begin{eqnarray*}
&&h_{\cal A} \leq \tfrac{\delta}{1-\delta}\sqrt{(\tfrac{s}{c_\tau})^{\frac{2}{2-\tau}}+1} \sqrt{T} (c_\tau\lambda) ^{\frac{1}{2-\tau}} < T^* \Rightarrow {\cal A}_{k+1} \subset {\cal A}^\dag,\\
&&(\rho s\lambda)^{\frac{1}{2-\tau}} - h_{\cal A} - T^* \geq (c_\tau\lambda) ^{\frac{1}{2-\tau}} \left(\rho(4-2\tau) - (2-\tau) - \tfrac{\delta}{1-\delta}\sqrt{T[(4- 2\tau)^2 +1]}\right) > 0\Rightarrow G_{\rho\lambda,s}\subset {\cal A}_{k+1}.
\end{eqnarray*}
This completes the proof of the lemma.
\end{proof}

Now we can give the proof of Theorem \ref{thm:convergence}.
\begin{proof}
For each $\lambda_k$-problem, we denote by ${\cal A}_{k,0}$ and ${\cal A}_{k,\diamond}$ the active set for the
initial guess and the last inner step (i.e., ${\cal A}(\lambda_k)$ in Algorithm \ref{alg:pdasc}), respectively.
Since $\lambda_0$ is large enough, we deduce that $G_{\lambda_1,s} =
\varnothing$  and  $G_{\lambda_1,s}\subset {\cal A}_{1,0}$. Then mathematics induction and by Lemma \ref{lem:monotone},
for any $k$ we have
\begin{equation}\label{equ:G}
  G_{\lambda_k,s} \subseteq {\cal A}_{k,0}\subset {\cal{A}}^\dag \quad\mbox{and}\quad G_{\rho\lambda_k,s}\subseteq {\cal A}_{k,\diamond}\subset {\cal{A}}^\dag.
\end{equation}
Therefore Algorithm \ref{alg:pdasc} is well-defined and when $k$ is large such that
\begin{equation*}
s \lambda_k< \left\{  \begin{array}{ll} \min\left\{|x_i^\dag|: x_i^\dag \neq 0 \right\} & \mbox{capped-}\ell^1, \mbox{SCAD}, \mbox{MCP}, \\
    (\min\left\{|x_i^\dag|: x_i^\dag \neq 0 \right\})^{2-\tau} & \mbox{bridge}, \end{array} \right.
\end{equation*}
we have ${\cal A}({\lambda_k}) = {\cal A}^\dag$ and hence Algorithm \ref{alg:pdas} converges in one step. To show the convergence of
the sequence of solutions to the true solution $x^\dag$, it suffices to check $\lim_{k\rightarrow \infty} p_{\cal A^\dag}(\lambda_k) = 0$. The convergence of $p$ follows
from the particular choice in Table \ref{tab:activeset} and its boundedness in \eqref{eqn:bound}. Hence we have
\begin{equation*}
x(\lambda_k)_{\cal A^\dag} = (\Psi_{\cal A^\dag}^t\Psi_{\cal A^\dag})^{-1}(\Psi_{\cal A^\dag}^t y - p_{\cal A^\dag}(\lambda_k)) \rightarrow x^\dag_{\cal A^\dag}.
\end{equation*}
This completes the proof of Theorem \ref{thm:convergence}.
\end{proof}

\end{document}